\documentclass[twoside]{article}
\usepackage[english]{babel}
\usepackage[latin1]{inputenc}
\usepackage{amsmath,amssymb}
\usepackage{amsfonts}
\usepackage{enumerate,psfrag,rotating}
\usepackage{comment}
\usepackage{latexsym}
\usepackage{graphics}
\usepackage{upref}
\usepackage{t1enc}
\usepackage{cite}
\usepackage[]{color}
\usepackage[T1]{fontenc} 
\usepackage{lmodern}  

\allowdisplaybreaks[1]

\numberwithin{equation}{section}

\newcommand{\F}{\ensuremath{F^s_{p,q}}}
\newcommand{\B}{\ensuremath{B^s_{p,q}}}

\newcommand{\Bd}{\ensuremath{\mathbf{B}^s_{p,q}}}
\newcommand{\Bdx}[3]{\ensuremath{\mathbf{B}^{#1}_{#2,#3}}}

\newcommand{\real}{\ensuremath{\mathbb{R}}}
\newcommand{\R}{\ensuremath{\mathbb{R}}}
\newcommand{\rn}{\ensuremath{\real^{n}}}
\newcommand{\nat}{\ensuremath{\mathbb{N}}}
\newcommand{\no}{\ensuremath{\nat_0}}
\newcommand{\non}{\ensuremath\mathbb{N}^n_0}
\newcommand{\zn}{\ensuremath{\mathbb{Z}^n}}
\newcommand{\comp}{\ensuremath{\mathbb{C}}}
\newcommand{\dint}{\mathrm{d}}

\newcommand{\bit}{\begin{itemize}}
\newcommand{\eit}{\end{itemize}}
\newcommand{\beq}{\begin{equation}}
\newcommand{\eeq}{\end{equation}}
\newcommand{\ds}{\displaystyle}
\newcommand{\supp}{\mathrm{supp}\,}

\newcommand{\Lip}{\mathrm{Lip}}

\newcommand\N{\mathbb{N}}

\newcommand\Z{\mathbb{Z}}

\newcommand\eps{\varepsilon}

\newcommand{\Bselfs}{\mathbf{B}^s_{p,q,\mathrm{selfs}}}
\newcommand{\selfs}{\mathrm{selfs}}

\newcommand{\ud}{\mathrm{d}}
\newcommand{\uD}{D}

\newtheorem{lemma}{Lemma}[section]
\newtheorem{proposition}[lemma]{Proposition}
\newtheorem{theorem}[lemma]{Theorem}
\newtheorem{corollary}[lemma]{Corollary}
\newtheorem{definition}[lemma]{Definition}
\newtheorem{rem}[lemma]{Remark}
\newcommand{\remark}[1]{\begin{rem}{\upshape #1}\end{rem}}
\newtheorem{example}[lemma]{Example}

\newcommand{\proofstart}{\mbox{P\,r\,o\,o\,f\, :\quad}}
\newcommand{\proofend}{\nopagebreak\hfill\raisebox{0.3em}{\fbox{}}\\}
\newenvironment{proof}{\proofstart}{\proofend}

\newcommand{\tr}{\mathrm{Tr}\,}

\newcommand{\Ex}{\mathrm{Ex}\,}

\begin{document}

\title{Non-smooth atomic decompositions, traces on Lipschitz domains, and pointwise multipliers in function spaces}
\author{Cornelia Schneider\footnote{Applied Mathematics III, University Erlangen--Nuremberg, Cauerstra\ss{}e 11, 91058 Erlangen,
Germany, email: {\tt schneider@am.uni-erlangen.de}} 
\ and Jan Vyb\'iral\footnote{Johann Radon Institute for Computational and
Applied Mathematics, Austrian Academy of Sciences, Altenbergerstra\ss e 69, A-4040 Linz, Austria,
email: {\tt  jan.vybiral@oeaw.ac.at} } \footnote{Corresponding author}}


\maketitle

\begin{abstract}
\noindent
We provide non-smooth atomic decompositions for Besov spaces $\Bd(\rn)$, $s>0$, $0<p,q\leq \infty$, defined via differences.
The results are used to compute the
trace of Besov spaces  on the boundary $\Gamma$ of bounded Lipschitz domains $\Omega$ with smoothness $s$
restricted to $0<s<1$ and no further restrictions on the parameters $p,q$. We conclude with some more
applications in terms of pointwise multipliers.\\

\noindent
\end{abstract}

\noindent
{\bf Math Subject Classifications (MSC2010):} 46E35, 42B35, 47B38.
\\

\noindent
{\bf Keywords and Phrases:} Lipschitz domains, Besov spaces, differences, real interpolation, atoms, traces, pointwise multipliers.

\section*{Introduction}

Besov spaces -- sometimes briefly denoted as B-spaces in the sequel -- of positive smoothness, have been investigated for many decades already,
resulting, for instance, from the study of partial differential equations, interpolation
theory, approximation theory, harmonic analysis. \\
There are several definitions of Besov spaces  $\Bd(\rn)$ to be found in the literature. Two of the most prominent approaches are the {\em Fourier-analytic approach} using Fourier transforms on the one hand and the {\em classical approach} via higher order differences involving the modulus of smoothness on the other. 
These two definitions are equivalent only with certain restrictions on the parameters, 
in particular, they differ for  $0<p<1$ and  $0<s\leq n(\frac1p-1)$, but may otherwise share similar
properties. \\
In the present paper we  focus on the {\em classical
  approach}, which introduces $\Bd(\rn)$ as those subspaces of $L_p(\rn)$ such
that 
\[
\|f|\Bd(\rn)\|_r = \|f|L_p(\rn)\| + \left(\int_0^1 t^{-sq} \omega_r(f,t)_p^q
\ \frac{\dint t}{t}\right)^{1/q}
\]
is finite, where $0<p,q\leq \infty$, $s>0$, $r\in\nat$ with $r>s$, and
$\omega_r(f,t)_p$ is the usual $r$-th modulus of smoothness of $f\in
L_p(\rn)$. \\
These spaces occur naturally in nonlinear approximation theory. Especially important is the case \mbox{$p<1$},  which is needed for the description of 
approximation classes of classical methods such as rational approximation and approximation by splines with free knots.  For more details we refer 
to the introduction of   \cite{dVP}.\\ 
For our purposes it will be convenient to use an equivalent characterization for the classical Besov spaces, cf. \cite{H-N}, \cite[Sect.~9.2]{T-F3}, 
and also \cite[Th.~2.11]{sch10},  relying on {\em smooth atomic decompositions}. They which allow us to characterize
$\Bd(\rn)$  as the space of those $f\in L_p(\rn)$ which can be represented as
\beq\label{intro-1}
f(x) = \sum_{j=0}^\infty \sum_{m\in\zn} \lambda_{j,m}
a_{j,m}(x),\quad x\in\rn,
\eeq
with the sequence of coefficients $\lambda = \{\lambda_{j,m} \in \comp:  j\in\no,
m\in\zn\}$ belonging to some appropriate sequence space $b^{s}_{p,q}$, where
$s>0$, $0<p,q\leq\infty$, and with smooth atoms $ a_{j,m}(x)$.\\

It is one of the  aims of the present paper to develop non-smooth atomic decompositions for Besov spaces $\Bd(\rn)$, cf. Theorem \ref{th:atom-dec} and Corollary \ref{th:atom-dec-3}. We will show that  one can relax the assumptions on the smoothness of the atoms $a_{j,m}$ used in the representation \eqref{intro-1} and, thus, replace these atoms with more general ones without loosing any crucial information 
compared smooth atomic decompositions for functions $f\in \Bd(\rn)$.  \\
There are only few forerunners dealing with non-smooth atomic decompositions in function spaces so far. 
We refer to the papers  \cite{tri03},\cite{MPP07}, and \cite{CL09},  all mainly considering the different Fourier-analytic
approach for Besov spaces and having in common that they restrict themselves to the technically simpler case when $p=q$.
Our  approach generalizes and extends these results and seems to be the first one covering the full range of indices 
$0<p,q\leq \infty$. The reader may also consult \cite{SchB11} for another generalization of the classical atomic decomposition
technique using building blocks of limited smoothness.\\

The additional freedom we gain in the choice of suitable non-smooth atoms $a_{j,m}$ for the atomic decompositions  
of $f\in \Bd(\rn)$ makes this approach well suited to further investigate Besov spaces $\Bd(\Omega)$ on non-smooth 
domains $\Omega$ and their boundaries $\Gamma$. In particular, we shall focus on  bounded Lipschitz domains  and start 
by  obtaining some interesting new properties concerning interpolation and equivalent quasi-norms for these spaces as 
well as an atomic decomposition for Besov spaces $\Bd(\Gamma)$, defined on the boundary $\Gamma=\partial\Omega$ of a Lipschitz domain.  \\

But the main goal of this article is to demonstrate the strength of the newly developed non-smooth atomic decompositions in view of trace results. 
The trace is taken with respect to the boundary $\Gamma$ of bounded Lipschitz domains $\Omega$. 
Our main result reads as 
\[
\tr \mathbf{B}^{s+\frac 1p}_{p,q}(\Omega)=\mathbf{B}^{s}_{p,q}(\Gamma), 
\]
where $n\geq 2$, $0<s<1$, and $0<p,q\leq\infty$, cf. Theorem \ref{B-trace-lip}. Its proof reveals how well suited non-smooth atoms 
are in order to tackle this problem. The limiting case $s=0$ is also  
considered in  Corollary \ref{B-trace-lim-n}. \\
In the range $0<s<1$, our results are optimal in the sense that 
there are no further restrictions on the parameters $p$, $q$.
The fact that we now also cover traces in Besov spaces $\Bd(\rn)$ with $p<1$ could 
be of particular interest in nonlinear approximation theory.\\
Moreover, as a by-product we obtain corresponding trace results on Lipschitz domains for Triebel-Lizorkin spaces, defined via atomic decompositions. \\
The papers \cite{sch08c} and \cite{sch10}, dealing with traces on hyperplanes and smooth domains, respectively, might be considered as forerunners of the trace  results established in this paper. Nevertheless, the methods we use now are completely different. \\ 
The same question for $s\ge 1$ was studied in \cite{JW84}. It turns out that in this case the function spaces on the boundary look
very different and also the extension operator must be changed. 
Moreover, based on the seminal work \cite{JK95}, traces on Lipschitz domains were studied in \cite[Th.~1.1.3]{May05} 
for the Fourier-analytic Besov spaces with the natural restrictions
\begin{equation}\label{eq:May}
(n-1)\max \left(\frac 1p-1,0\right)<s<1\quad \text{ and  }\quad \frac{n-1}{n}<p.
\end{equation}
Our Theorem \ref{B-trace-lip} actually covers and extends \cite[Th.~1.1.3]{May05}, as for the parameters restricted by \eqref{eq:May}
the Besov spaces defined by differences coincide with the Fourier-analytic Besov spaces. 

In contrast to {\sc Mayboroda} we make use of the classical Whitney extension operator 
and the cone property of Lipschitz domains in order 
to establish our results instead of potential layers and interpolation. Moreover, the extension operator we construct 
is not linear -- and in fact cannot be whenever $s<(n-1)\max (\frac 1p-1,0)$ -- compared to the extension operator 
in \cite[Th.~1.1.3]{May05}. Let us recall that the importance of non-linear extension operators is known in the theory
of differentiable spaces since the pioneering work of Gagliardo \cite{Ga57}, cf. also \cite[Chapter 5]{Bu98}.


Finally, we shall use the non-smooth atomic decompositions again to deal with pointwise multipliers in the respective function spaces. Let $\Bselfs(\rn)$ denote the self-similar spaces introduced in Definition \ref{def-Bselfs} and $M(\Bd(\rn))$ the set of all pointwise multipliers of $\Bd(\rn)$. We prove for  $s>0$, $0<p,q\leq \infty$  in Theorem \ref{th:multipliers-1}  the relationship 
\beq\label{intro-2} 
\bigcup_{\sigma>s}\mathbf{B}^{s}_{p,q,\selfs}(\rn)\subset M(\Bd(\rn))\hookrightarrow \mathbf{B}^{s}_{p,q,\selfs}(\rn).
\eeq
Additionally, if $0<p\leq 1$, one even has  a coincidence in terms of 
$\ 
M(\mathbf{B}^s_{p,p}(\rn))=\mathbf{B}^{s}_{p,p,\selfs}(\rn). 
\ $
Our results generalize the multiplier assertions from \cite{tri03} to the case when $p\neq q$. Moreover, they extend previous results to classical Besov spaces with small parameters $s$ and $p$. In this context we refer to \cite{mazya}, \cite{MaSh85}, and \cite{MaSh09}, where pointwise multipliers in Besov spaces with $p,q\geq 1$ and $p=q$ were studied in detail. \\
We conclude using \eqref{intro-2} in order to discuss under which circumstances the characteristic function $\chi_{\Omega}$ of a bounded domain $\Omega$  in $\rn$ is a pointwise multiplier in $\Bd(\rn)$ -- establishing a connection between pointwise multipliers and  
certain fundamental notion of fractal geometry, so-called $h$-sets, cf. Definition \ref{def-h-set}. In particular, 
if a boundary $\Gamma=\partial \Omega$  is an $h$-set satisfying
$$
\sup_{j\in\nat_0}\sum_{k=0}^{\infty}2^{k\sigma q}\left(\frac{h(2^{-j})}{h(2^{-j-k})}2^{-kn}\right)^{q/p}<\infty,
$$
where  $\sigma>0$, $0<p<\infty$, and $0<q\leq\infty$, then Theorem \ref{th:multipliers-2} shows  that
\[
\chi_{\Omega}\in \mathbf{B}^{\sigma}_{p,q,\selfs}(\rn).
\]


The present paper is organized as follows: Section 1 contains notation, definitions, and preliminary assertions on smooth atomic decompositions.
The main  investigation starts in  Section 2, where we construct non-smooth atomic decompositions for the spaces under focus. Afterwards Section 3 
provides new insights  (and helpful results) concerning function spaces on Lipschitz domains and their boundaries. These powerful techniques are 
then used in Section 4 in order to compute traces on Lipschitz domains --  the heart of this article. Finally, we conclude with some further 
applications of non-smooth atomic decompositions in terms of pointwise multipliers in Section 5.\\


\section{Preliminaries}

We use standard notation. Let $\nat$ be the collection of all natural numbers
and let $\no = \nat \cup \{0 \}$. Let $\rn$ 
be euclidean $n$-space, $n \in \nat$, $\comp$ the complex plane. The set of
multi-indices $\beta=(\beta_1, \dots, \beta_n)$, $\beta_i\in\no$, $i=1, \dots,
n$, is
denoted by $\non$, with $|\beta|=\beta_1 + \cdots + \beta_n$, as
usual. Moreover, if $x=(x_1, \dots, x_n)\in\rn$ and $\beta=(\beta_1, \dots,
\beta_n)\in\non$ we put $x^\beta=x_1^{\beta_1} \cdots x_n^{\beta_n}$. \\
We use the symbol \ '$\lesssim$'\  in 
\[
a_k \lesssim b_k \quad \mbox{or} \quad \varphi(x) \lesssim \psi(x)
\]
always to mean that there is a positive number $c_1$ 
such that
\[
a_k \leq c_1\,b_k  \quad \mbox{or} \quad
\varphi(x) \leq c_1\,\psi(x) 
\]
for all admitted values of the discrete variable $k$ or the
continuous variable $x$, where $\{a_k\}_k$, $\{b_k\}_k$ are
non-negative sequences and $\varphi$, $\psi$ are non-negative
functions. 
We use the equivalence \ `$\sim$' \ in
\[a_k \sim b_k \quad \mbox{or} \quad \varphi(x) \sim \psi(x)\]
for 
\[
a_k\lesssim b_k\quad \text{and} \quad b_k\lesssim a_k \qquad \text{or}\qquad 
\varphi(x) \lesssim \psi(x)\quad \text{and} \quad \psi(x) \lesssim \varphi(x).
\]
If $a\in\real$, then $a_+ := \max(a,0)$ and $[a]$
denotes the integer part of $a$.\\
Given two (quasi-) Banach spaces $X$ and $Y$, we write $X\hookrightarrow Y$
if $X\subset Y$ and the natural embedding of $X$ into $Y$ is continuous. 
All unimportant positive constants will be denoted by $c$, occasionally with
subscripts. For convenience, let both $ \dint x $ and $ |\cdot| $ stand for the
($n$-dimensional) Lebesgue measure in the sequel. 
$L_p(\rn)$, with $0<p\leq\infty$, stands for the usual quasi-Banach space with respect to the Lebesgue measure, quasi-normed by
\[
\|f|L_p(\rn)\|:=\left(\int_{\rn}|f(x)|^p\ud x\right)^{\frac 1p}
\]
with the appropriate modification if $p=\infty$. Throughout the paper  $\Omega$ will denote a domain in $\rn$ and  the Lebesgue space  $L_p(\Omega)$ is defined in the usual way. \\

We denote by $C^K(\rn)$ the space of all $K$-times continuously differentiable functions $f:\R^n \to\R$
equipped with the norm
$$
\|f|C^K(\R^n)\|=\max_{|\alpha|\le K} \sup_{x\in\R^n} |D^\alpha f(x)|.
$$
Additionally,  $C^{\infty}(\rn)$ contains the set of smooth and bounded functions on $\rn$, i.e., 
$$\ 
C^{\infty}(\rn):=\bigcap_{K\in\nat}C^K(\rn),
\ $$
whereas $C^{\infty}_0(\rn)$ denotes the space of smooth functions with compact support. 

Furthermore, $B(x_0,R)$ stands for an open ball with radius $R>0$ around $x_0\in\rn$, 
\beq\label{op_ba}
B(x_0,R)=\{x\in\rn:\, |x-x_0|<R\}.
\eeq
Let $Q_{j,m}$ with $j\in\no$ and $m\in\zn$ denote a cube in
$\rn$ with sides parallel to the axes of coordinates, centered at
$2^{-j}m$, and with side length $2^{-j+1}$. For a cube $Q$ in $\rn$ and $r>0$,
we denote by  $rQ$  the cube in $\rn$ concentric with $Q$ and with side length
$r$ times the side length of $Q$. 
Furthermore, $\chi_{j,m}$ stands for the characteristic function of $Q_{j,m}$.\\

Let $G\subset \rn$ and $j\in\N_0$. We use the abbreviation
\beq
{\sum_{m\in\zn}}^{G,j}=\sum_{m\in\mathbb{Z}^n, Q_{j,m}\cap G\neq\emptyset},
\eeq
where $G$ will usually denote either a domain $\Omega$ in $\rn$ or its boundary $\Gamma$.


\subsection{Smooth atomic decompositions in function spaces}

We introduce the Besov  spaces $\Bd(\Omega)$ through their decomposition properties.
This provides a constructive definition expanding functions $f$ via smooth atoms (excluding any moment conditions) and suitable coefficients, 
where the latter belong to certain sequence spaces denoted by $b^s_{p,q}(\Omega)$ defined below. 


\begin{definition}\label{def-seq-dom}
Let $0<p,q\leq\infty$, $s\in\real$. Furthermore, let $\Omega\subset\rn$ and $\lambda=\{\lambda_{j,m}\in\mathbb{C}:j\in\mathbb{N}_0, m\in\zn\}$. 
Then\index{sequence spaces!of type $b^s_{p,q}(\Omega)$}
\[
b^s_{p,q}(\Omega)=\left\{\lambda: \|\lambda|b^s_{p,q}(\Omega)\|=\left(\sum_{j=0}^{\infty}2^{j(s-\frac np)q}\left({\sum_{m\in\zn}}^{\Omega,j}|\lambda_{j,m}|^p\right)^{q/p}\right)^{1/q}<\infty\right\}
\]
$($with the usual modification if $p=\infty$ and/or $q=\infty$$)$.  
\end{definition}

\remark{If $\Omega=\rn$, we simply write 
 $b^s_{p,q}$ and $\sum_{m}$  instead of  $b^s_{p,q}(\Omega)$ and ${\sum_m}^{\Omega,j}$, respectively.}

Now we define the smooth  atoms.

\begin{definition}\label{def-atoms}
 Let $K\in\mathbb{N}_0$ and $d>1$. A $K$-times continuously differentiable complex-valued function $a$ on $\rn$ $($continuous if $K=0$$)$ 
 is called a $K$-atom if for some $j\in\mathbb{N}_0$
\beq\label{supp_1}
\supp a\subset dQ_{j, m}\quad \text{for some } m\in\zn,
\eeq
and 
\beq\label{smooth}
|\uD^\alpha a(x)|\leq 2^{|\alpha|j} \quad\text{ for } |\alpha|\leq K.
\eeq
\end{definition}

It is convenient to write $a_{j,m}(x)$ instead of $a(x)$ if this atom is located at $Q_{j,m}$ according to \eqref{supp_1}. Furthermore, $K$ denotes the smoothness of the atom, cf.~\eqref{smooth}. \\

We define Besov spaces $\Bd(\Omega)$ using the {\em atomic approach}. 

\smallskip

\begin{definition}\label{intr_char}
 Let $s>0$ and $0<p,q\leq\infty$. Let $d>1$ and $K\in\mathbb{N}_0$  with
\[
K\geq (1+[s]) 
\]
be fixed. Then $f\in L_p(\Omega)$ belongs to $\Bd(\Omega)$ if, and only if, it can be represented as
\beq\label{repr_01'}
f(x)=\sum_{j=0}^\infty{\sum_{m\in\zn}}^{\Omega,j}\lambda_{j,m}a_{j,m}(x),
\eeq
where the $a_{j,m}$ are $K$-atoms $(j\in\mathbb{N}_0)$ with
\[
\supp a_{j,m}\subset dQ_{j,m},\qquad j\in\mathbb{N}_0, \quad m\in\zn,
\]
and $\lambda\in b^{s}_{p,q}(\Omega)$, convergence being in $L_p(\Omega)$. Furthermore,
\beq\label{norm-inf1}
\|f|\Bd(\Omega)\|:=\inf\|\lambda|b^s_{p,q}(\Omega)\|,
\eeq
where the infimum is taken over all admissible representations \eqref{repr_01'}. \\
\end{definition}

\remark{\label{extr_char} 
According to \cite{T-F3}, based on \cite{H-N}, the above defined spaces  are independent of $d$ and $K$. 
This may justify our omission of $K$ and $d$ in \eqref{norm-inf1}. \\
Since the atoms $a_{j,m}$ used in Definition \ref{intr_char} are defined also outside of $\Omega$, the spaces $\Bd(\Omega)$ can as well be regarded as restrictions of the corresponding spaces on $\rn$ in the usual interpretation, i.e.,
\[
\Bd(\Omega)=\{f\in L_p(\Omega): \quad\text{there exists}\quad g\in \Bd(\rn) \quad\text{with}\quad  g\big|_\Omega=f\},
\]
furnished with the norm
\[
\|f|\Bd(\Omega)\|=\inf \left\{ \|g|\Bd(\rn)\| \quad \text{with}\quad g\big|_\Omega=f\right\},
\]
where $g\big|_\Omega=f$ denotes the restriction of $g$ to $\Omega$.  
Therefore, well-known embedding results for B-spaces defined on $\rn$ carry over to those defined on domains $\Omega$. Let $s>0$, $\varepsilon>0$, $0<q,u\leq\infty$, and $q\leq v\leq \infty$. Then we have
\[
\mathbf{B}^{s+\varepsilon}_{p,u}(\Omega)\hookrightarrow \Bd(\Omega)\qquad \text{and}\qquad \Bd(\Omega)\hookrightarrow \mathbf{B}^s_{p,u}(\Omega),
\]
cf. \cite[Th.~1.15]{HS08}, where also further embeddings for Besov spaces may be found.
}
\smallskip

\paragraph{Classical approach}
Originally Besov spaces were defined merely using higher order differences instead of atomic decompositions. The question arises whether this {\em classical approach} coincides with our {\em atomic approach}. This might not always be the case but is true for spaces defined on $\rn$ and on  so-called  $(\varepsilon,\delta)$-domains which we introduce next. \\

Recall that domain always stands for open set. The boundary of $\Omega$ is denoted by $\Gamma=\partial\Omega$.

\begin{definition}\label{def-dom}
Let $\Omega$ be a domain in $\rn$ with $\Omega\neq \rn$. Then $\Omega$ is said to be an $(\varepsilon,\delta)$-domain\index{domains!$(\varepsilon,\delta)$-domain}, where $0<\varepsilon<\infty$ and $0<\delta<\infty$, if it is connected and if for any $x\in \Omega$, $y\in\Omega$ with $|x-y|<\delta$ there is a curve $L\subset\Omega$, connecting $x$ and $y$ such that $|L|\leq \varepsilon^{-1}|x-y|$ and 
\beq\label{eps-dom-cond}
\mathrm{dist}(z,\Gamma)\geq \varepsilon \min(|x-z|,|y-z|),\qquad z\in L.
\eeq 
\end{definition}

\remark{
All domains we will be concerned with in the sequel are  $(\varepsilon, \delta)$-domains. In particular, the definition includes \emph{minimally smooth} domains in the sense of Stein, cf. \cite[p.~189]{stein}, and therefore bounded Lipschitz domains (as will be considered in Section \ref{sec-3}). \\
Furthermore, the half space
$\ \rn_+~:=\{x:x=(x',x_n)\in\rn, x'\in\mathbb{R}^{n-1}, x_n>0\}\ $ is another example.

}

It is well-known that $(\varepsilon, \delta)$-domains play a crucial role concerning questions of extendability. It is precisely this property which was used in \cite[Th.~2.10]{sch10} to show that for $(\varepsilon, \delta)$-domains the atomic approach for B-spaces is equivalent to the {\em classical approach} (in terms of equivalent quasi-norms), 
which introduces $\Bd(\Omega)$ as the subspace of $L_p(\Omega)$ such that 
\beq\label{class-B-dom}
\|f|\Bd(\Omega)\|_r = \|f|L_p(\Omega)\| + \left(\int_0^1 t^{-sq} \omega_r(f,t,\Omega)_p^q
\ \frac{\dint t}{t}\right)^{1/q}
\eeq
is finite, where $0<p,q\leq \infty$ $($with the usual modification if $q=\infty$$)$, $s>0$, $r\in\nat$ with $r>s$. Here $\omega_r(f,t,\Omega)_p$ stands for the usual $r$-th modulus of smoothness of a
function $f\in L_p(\Omega)$,
\beq\label{def-Bd-dom}
\omega_r(f,t,\Omega)_p=\sup_{|h|\leq t} \|\Delta_h^r f(\cdot, \Omega)\mid L_p(\Omega)\|, \quad
t>0,
\eeq
where 
\begin{equation}\label{eq:finalJV2}
\Delta^r_hf(x,\Omega):=
\begin{cases}
 \Delta^r_h f(x),& x,x+h,\dots,x+rh\in\Omega,\\
 0, & \text{otherwise},
\end{cases}
\end{equation}
This approach for the spaces $\Bd(\Omega)$ was used in \cite{dVS93}.
The proof of the coincidence uses the fact that the classical and atomic approach can be identified for spaces defined on $\rn$, which follows from results by {Hedberg, Netrusov}
\cite{H-N} on atomic decompositions and by {Triebel}\cite[Section~9.2]{T-F3} on the reproducing formula. \\


The classical scale of Besov spaces contains many well-known function
spaces. For example, if $p=q=\infty $, one recovers the H\"older-Zygmund spaces $\ {\mathcal C}^s(\rn)$, i.e., 
\begin{equation}\label{B=Zyg} 
 \Bdx{s}{\infty}{\infty}(\rn) =  {\mathcal C}^s(\rn), \quad s>0.
\end{equation}

Later on we will need the following homogeneity estimate proved recently in \cite[Th.~2]{schvyb11a} based on \cite{CLT07}.

\begin{theorem}\label{hom-B}
Let $0<\lambda\leq 1$ and $f\in \mathbf{B}^s_{p,q}(\R^n)$ with $\supp f\subset B(0,\lambda)$. Then
\beq\label{hom-B-eq}
\|f(\lambda \cdot)|\Bd(\R^n)\|\sim \lambda^{s-n/p}\|f|\Bd(\R^n)\|.
\eeq
\end{theorem}

\section{Non-smooth atomic decompositions}
\label{sect-2}

Our aim is to provide a {non-smooth} atomic characterization of Besov  spaces $\Bd(\rn)$, i.e., relaxing the assumptions 
about the smoothness of the atoms $a_{j,m}$ in Definition \ref{def-atoms}. Note that condition \eqref{smooth} is equivalent to
\begin{equation}\label{eq:dil1}
\|a(2^{-j}\cdot)|C^K(\rn)\|\leq 1.
\end{equation}
We replace the $C^K$-norm with $K>s$ by a Besov quasi-norm $\mathbf{B}^{\sigma}_{p,p}(\R^n)$ with $\sigma>s$ or in case of $0<s<1$ by a norm in the space of Lipschitz functions $\Lip(\R^n)$.

The following non-smooth atoms were introduced in \cite{tri02}.
They will be very adequate when considering (non-smooth) atomic decompositions of spaces defined on Lipschitz domains (or on the 
boundary of a Lipschitz domain, respectively).
\begin{definition}\label{Lip-atom}
\begin{enumerate}
\item[(i)] The space of Lipschitz functions $\Lip(\R^n)$ is defined as the collection of all real-valued functions $f:\R^n\to\R$ such that
$$
\|f|\Lip(\R^n)\|=\max\left\{\sup_{x}|f(x)|, \quad \sup_{x\not =y}\frac{|f(x)-f(y)|}{|x-y|}\right\}<\infty.
$$
\item[(ii)]
We say that $a\in \Lip(\rn)$ is a $\Lip$-atom, if for some $j\in\N_0$
\beq\label{lip-atom-1}
\supp a\subset d Q_{j,m}, \quad m\in\zn,\; d>1,
\eeq
and
\beq\label{lip-atom-2}
|a(x)|\leq 1, \qquad |a(x)-a(y)|\leq 2^{j}|x-y|. 
\eeq
\end{enumerate}
\end{definition}

\remark{
One might  use alternatively in \eqref{lip-atom-2} that 
\begin{equation}\label{eq:dil2}
\|a(2^{-j}\cdot)|\Lip(\rn)\|\leq 1.
\end{equation}
}

We use the abbreviation 
$$ 
\mathbf{B}^s_p(\rn)=\mathbf{B}^s_{p,p}(\rn)\quad \text{with} \quad 0<p\leq\infty, \quad s>0.\ 
$$
In particular, in view of \eqref{B=Zyg},
\begin{equation*}
{\mathcal C}^{s}(\rn)=\mathbf{B}^s_{\infty}(\rn), \quad s>0,
\end{equation*}
are the H\"older-Zygmund spaces. 

\begin{definition}\label{def-atoms-ns2}
Let $0<p\leq\infty$,  $\sigma>0$ and $d>1$. Then  $a\in \mathbf{B}^{\sigma}_{p}(\rn)$ is called a $(\sigma,p)$-atom if 
for some $j\in\nat_0$
\beq\label{atom-ns-01}
\supp a\subset dQ_{j,m}\qquad \text{for some } m\in\zn, 
\eeq
and 
\beq\label{atom-ns-02}
\|a(2^{-j}\cdot)|\mathbf{B}^{\sigma}_p(\rn)\|\leq 1.
\eeq
\end{definition}

\remark{Note that if $\sigma<\frac np$ then $(\sigma,p)$-atoms might be unbounded.  Roughly speaking, they arise by dilating $\mathbf{B}^{\sigma}_p$-normalized functions. 
Obviously, the condition \eqref{atom-ns-02} is a straightforward modification of \eqref{eq:dil1} and \eqref{eq:dil2}.\\

In general, it is convenient to write $a_{j,m}(x)$ instead of $a(x)$ if the atoms are located at $Q_{j,m}$ according 
to \eqref{lip-atom-1} and \eqref{atom-ns-01}, respectively. Furthermore, $\sigma$ denotes the 'non-smoothness' of the 
atom, cf.~\eqref{smooth}.  \\

The non-smooth atoms we consider in Definition \ref{def-atoms-ns2}, are  renormalized versions 
of the non-smooth $(s,p)^{\sigma}$-atoms
considered in \cite{tri03} and \cite{triw96}, where \eqref{atom-ns-02} is replaced by
\[
a\in B_p^{\sigma}(\rn)\qquad \text{with}\qquad \|a(2^{-j}\cdot)|B_p^{\sigma}(\rn)\|\leq 2^{j(\sigma-s)},
\]
resulting in corresponding changes concerning the definition of the sequence spaces $b^s_{p,q}$ used for the atomic decomposition.\\
However, the function spaces we consider are different from the ones considered there. Furthermore, for our purposes 
(studying traces later on) it is convenient to shift the factors $2^{j(s-\frac np)}$ to the sequence spaces. 
}

We wish to compare these atoms with the smooth atoms in Definition \ref{def-atoms}. 

\begin{proposition}\label{prop-atom-ns}
Let $0<p\leq \infty$ and $0<\sigma<K$. Furthermore, let $d>1$, $j\in\nat_0$, and $m\in\zn$.
Then any $K$-atom $a_{j,m}$ is a $(\sigma,p)$-atom.
\end{proposition}

\proofstart
Since the functions $a_{j,m}(2^{-j}\cdot)$ have compact support, we obtain
\[
\|a_{j,m}(2^{-j}\cdot)|\mathbf{B}^{\sigma}_{p}(\rn)\|\lesssim \|a_{j,m}(2^{-j}\cdot)|C^{K}(\rn)\|\leq 1,
\]
with constants independent of $j$, giving the desired result for  non-smooth atoms from Definition \ref{def-atoms-ns2}.
\proofend

The use of atoms with limited smoothness (i.e. finite element functions or splines) 
was studied already in \cite{Os94}, where the author deals with spline approximation (and traces) in Besov spaces.\\

The following theorem contains the main result of this section. It gives the counterpart of Definition \ref{intr_char} 
and provides a non-smooth atomic decomposition of the spaces $\Bd(\rn)$.

\begin{theorem}\label{th:atom-dec}
 Let $0<p,q\leq\infty$, $0<s<\sigma$, and  $d>1$.  
Then $f\in L_p(\rn)$ belongs to $\Bd(\rn)$ if, and only if, it can be represented as
\beq\label{repr_01}
f=\sum_{j=0}^\infty\sum_{m\in\zn}\lambda_{j,m}a_{j,m},
\eeq 
where the $a_{j,m}$ are $(\sigma,p)$-atoms $(j\in\mathbb{N}_0)$ with
$\supp a_{j,m}\subset dQ_{j,m}$,  $j\in\mathbb{N}_0$, $m\in\zn,$
and $\lambda\in b^s_{p,q}$, convergence being in $L_p(\rn)$. Furthermore,
\beq\label{norm-inf1'}
\|f|\Bd(\rn)\|=\inf\|\lambda|b^s_{p,q}\|,
\eeq
where the infimum is taken over all admissible representations \eqref{repr_01}.
\end{theorem}

\proofstart
We have the atomic decomposition based on smooth $K$-atoms according to Definition \ref{intr_char}. 
By Proposition \ref{prop-atom-ns} classical 
$K$-atoms are special $(\sigma,p)$-atoms. Hence, it is enough to prove that 
\begin{equation}\label{eq:todo1}
\|f|\mathbf{B}^s_{p,q}(\R^n)\|\lesssim \left(\sum_{k=0}^\infty2^{k(s-\frac np)q}\left(\sum_{l\in\Z^n}|\lambda_{k,l}|^p\right)^{q/p}\right)^{1/q}
\end{equation}
for any atomic decomposition
\beq\label{h-0}
f=\sum_{k=0}^\infty \sum_{l\in\Z^n}\lambda_{k,l}a^{k,l},
\eeq
where $a^{k,l}$ are $(\sigma,p)$-atoms according to Definition \ref{def-atoms-ns2}.   

For this purpose we expand each function $a^{k,l}(2^{-k}\cdot)$ optimally in $\mathbf{B}^{\sigma}_p(\R^n)$ with respect to classical 
$K$-atoms $b^{j,w}_{k,l}$ where $\sigma<K$, 
\beq\label{h-3}
a^{k,l}(2^{-k}x)=\sum_{j=0}^{\infty}\sum_{w\in\zn}\eta_{j,w}^{k,l}b_{k,l}^{j,w}(x), \quad x\in\rn,
\eeq
with
\beq\label{h-1}
\supp b_{k,l}^{j,w}\subset Q_{j,w}, \qquad \left|\uD^{\alpha}b_{k,l}^{j,w}(x)\right|\leq 2^{|\alpha|j}, \quad |\alpha|\leq K,
\eeq
and 
\beq\label{eq:have2}
\left(\sum_{j=0}^\infty 2^{j(\sigma-\frac np)p}\sum_{w\in\Z^n}|\eta^{k,l}_{j,w}|^p\right)^{\frac 1p}=
\|\eta^{k,l}|b^{\sigma}_{p,p}\|\sim \|a^{k,l}(2^{-k}\cdot)|\mathbf{B}^{\sigma}_p(\rn)\|\lesssim 1.
\eeq
Hence,
$$
a^{k,l}(x)=\sum_{j=0}^\infty \sum_{w\in\Z^n}\eta^{k,l}_{j,w}b^{j,w}_{k,l}(2^kx),
$$
where the functions $b^{j,w}_{k,l}(2^k\cdot)$ are supported by cubes with side lengths $\sim 2^{-k-j}$. By \eqref{h-1} we have
\[\left|\uD^{\alpha}b^{j,w}_{k,l}(2^kx)\right|=2^{k|\alpha|}\left|(\uD^{\alpha}b^{j,w}_{k,l})(2^kx)\right|\leq 2^{(j+k)|\alpha|}.
\]
Replacing $j+k$ by $j$ and putting $d^{j,w}_{k,l}(x):=b^{j-k,w}_{k,l}(2^kx)$, we obtain that 
\beq\label{h-2}
a^{k,l}(x)=\sum_{j=k}^\infty \sum_{w\in\Z^n}\eta^{k,l}_{j-k,w}
d^{j,w}_{k,l}(x),
\eeq
where $d^{j,w}_{k,l}$ are classical $K$-atoms supported by cubes with side lengths $\sim 2^{-j}$. 
We insert \eqref{h-2} into the expansion \eqref{h-0}. We fix $j\in\nat_0$ and $w\in\zn$, and collect all non-vanishing terms $d_{k,l}^{j,w}$ in the expansions \eqref{h-2}. We have $k\leq j$. Furthermore, multiplying 
\eqref{h-3} if necessary with suitable cut-off functions it follows that there is a natural number $N$ such that for fixed $k$ only at most $N$ points $l\in\zn$ contribute to $d^{j,w}_{k,l}$. We denote this set by $(j,w,k)$. Hence its cardinality is at most $N$, where $N$ is independent of $j,w,k$. Then 
\[
d^{j,w}(x)=\frac{\sum_{k\leq j}\sum_{l\in(j,w,k)}\eta^{k,l}_{j-k,w}\cdot\lambda_{k,l}\cdot d^{j,w}_{k,l}(x)}{\sum_{k\leq j}\sum_{l\in(j,w,k)}|\eta^{k,l}_{j-k,w}|\cdot|\lambda_{k,l}|}
\]
are correctly normalized smooth $K$-atoms located in cubes with side lengths $\sim 2^{-j}$ and centered at $2^{-j}w$.
Let 
\begin{equation}\label{eq:have1}
\nu_{j,w}=\sum_{k\le j} \sum_{l\in(j,w,k)}|\eta^{k,l}_{j-k,w}|\cdot|\lambda_{k,l}|.
\end{equation}
Then we obtain a classical atomic decomposition in the sense of Definition \ref{intr_char}
\[
f=\sum_{j}\sum_w \nu_{j,w}d^{j,w}(x),
\]
where $d^{j,w}$ are $K$-atoms and 
$$
\|f|\mathbf{B}^s_{p,q}(\R^n)\|\lesssim \|\nu|b^s_{p,q}\|.
$$
Therefore, in order to prove \eqref{eq:todo1}, it is enough to show, that
\begin{equation}\label{eq:todo2}
\|\nu|b^s_{p,q}\|\lesssim \|\lambda|b^s_{p,q}\|
\end{equation}
if \eqref{eq:have2} holds.

Let $0<\eps<\sigma-s$. Then we obtain by \eqref{eq:have1} that (assuming $p<\infty$) 
\begin{equation}\label{eq:have3}
|\nu_{j,w}|^p\lesssim \sum_{k\le j}\sum_{l\in(j,w,k)}2^{(j-k)p\varepsilon}
|\eta_{j-k,w}^{k,l}|^p|\lambda_{k,l}|^p,
\end{equation}
where we used the bounded cardinality of the sets $(j,w,k)$.

This gives for $q/p\le 1$
{\allowdisplaybreaks
\begin{align*}
\|\nu|b^s_{p,q}\|^q&= \sum_{j=0}^\infty2^{j(s-n/p)q}\left(\sum_{w\in\Z^n} |\nu_{j,w}|^p\right)^{q/p}\\
&\lesssim \sum_{j=0}^\infty2^{j(s-n/p)q}\left(\sum_{w\in\Z^n}
\sum_{k=0}^j \sum_{l\in(j,w,k)}2^{(j-k)p\varepsilon}|\eta_{j-k,w}^{k,l}|^p|\lambda_{k,l}|^p
\right)^{q/p}\\
&\le\sum_{j=0}^\infty 2^{j(s-n/p)q}\sum_{k=0}^j\left(\sum_{w\in\Z^n}
 \sum_{l\in(j,w,k)}2^{(j-k)p\varepsilon} |\eta_{j-k,w}^{k,l}|^p|\lambda_{k,l}|^p
\right)^{q/p}\\
&=\sum_{k=0}^\infty \sum_{j=k}^\infty 2^{j(s-n/p)q}\left(\sum_{w\in\Z^n}
 \sum_{l\in(j,w,k)}2^{(j-k)p\varepsilon}|\eta_{j-k,w}^{k,l}|^p|\lambda_{k,l}|^p
\right)^{q/p}\\
&=\sum_{k=0}^\infty \sum_{j=0}^\infty 2^{(j+k)(s-n/p)q}\left(\sum_{w\in\Z^n}
\sum_{l\in(j+k,w,k)}2^{jp\varepsilon}|\eta_{j,w}^{k,l}|^p|\lambda_{k,l}|^p
\right)^{q/p}\\
&=\sum_{k=0}^\infty 2^{k(s-n/p)q} \sum_{j=0}^\infty 2^{j(s-\sigma+\varepsilon)q}\left(\sum_{w\in\Z^n}
\sum_{l\in(j+k,w,k)}2^{j(\sigma-n/p)p}|\eta_{j,w}^{k,l}|^p|\lambda_{k,l}|^p
\right)^{q/p}\\
&\lesssim\sum_{k=0}^\infty 2^{k(s-n/p)q} \left(\sum_{j=0}^\infty \sum_{w\in\Z^n} \sum_{l\in(j+k,w,k)}
2^{j(\sigma-n/p)p}|\eta_{j,w}^{k,l}|^p|\lambda_{k,l}|^p
\right)^{q/p}\\
&\le\sum_{k=0}^\infty 2^{k(s-n/p)q} \left(\sum_{j=0}^\infty \sum_{w\in\Z^n} \sum_{l\in \Z^n}
2^{j(\sigma-n/p)p}|\eta_{j,w}^{k,l}|^p|\lambda_{k,l}|^p
\right)^{q/p}\\
&=\sum_{k=0}^\infty 2^{k(s-n/p)q} \left(\sum_{l\in\Z^n} |\lambda_{k,l}|^p \sum_{j=0}^\infty \sum_{w\in\Z^n}
2^{j(\sigma-n/p)p}|\eta_{j,w}^{k,l}|^p
\right)^{q/p}\\
&\lesssim\sum_{k=0}^\infty 2^{k(s-n/p)q} \left(\sum_{l\in\Z^n} |\lambda_{k,l}|^p\right)^{q/p}=\|\lambda|b^s_{p,q}\|^q.
\end{align*}}
We have used \eqref{eq:have2} in the last inequality.

If $q/p>1$, we shall use the following inequality, which holds for every non-negative sequence 
$\{\gamma_{j,k}\}_{0\le k\le j<\infty}$, every $\alpha\ge 1$ and every $\varepsilon>0$.
\begin{equation}\label{eq:todo3}
\sum_{j=0}^\infty \left(\sum_{k=0}^j2^{-(j-k)\varepsilon}\gamma_{j,k}\right)^\alpha
\le c_{\alpha,\varepsilon} \sum_{k=0}^\infty \left(\sum_{j=k}^\infty \gamma_{j,k}\right)^\alpha.
\end{equation}
If $\alpha=\infty$, \eqref{eq:todo3} has to be modified appropriately.
To prove \eqref{eq:todo3} for $\alpha<\infty$, we use H\"older's inequality and the embedding $\ell_1\hookrightarrow \ell_\alpha$
\begin{align*}
\sum_{j=0}^\infty \left(\sum_{k=0}^j2^{-(j-k)\varepsilon}\gamma_{j,k}\right)^\alpha
&\le \sum_{j=0}^\infty \left(\sum_{k=0}^j 2^{-(j-k)\varepsilon\alpha'}\right)^{\alpha/\alpha'}
\left(\sum_{k=0}^j \gamma_{j,k}^\alpha\right)^{\alpha/\alpha}\\
&\lesssim \sum_{j=0}^\infty \sum_{k=0}^j \gamma_{j,k}^\alpha = \sum_{k=0}^\infty \sum_{j=k}^\infty \gamma_{j,k}^\alpha
\le \sum_{k=0}^\infty \left(\sum_{j=k}^\infty\gamma_{j,k}\right)^\alpha.
\end{align*}
We use \eqref{eq:have3} and \eqref{eq:todo3} with $p(\sigma-s-\varepsilon)$ instead of $\varepsilon$ and $\alpha=q/p> 1$,
{\allowdisplaybreaks
\begin{align*}
\|\nu|b^s_{p,q}\|^q&\lesssim \sum_{j=0}^\infty2^{j(\sigma-\frac np)q}\left(\sum_{w\in\Z^n}
\sum_{k=0}^j \sum_{l\in(j,w,k)}2^{(j-k)p\varepsilon}
|\eta_{j-k,w}^{k,l}|^p|\lambda_{k,l}|^p
\right)^{q/p}\\
&=\sum_{j=0}^\infty \left(\sum_{k=0}^j2^{-(j-k)p(\sigma-s-\varepsilon)}
\sum_{w\in\Z^n} \sum_{l\in(j,w,k)} 2^{k(s-n/p)p}2^{(j-k)(\sigma-\frac np)p}|\eta_{j-k,w}^{k,l}|^p|\lambda_{k,l}|^p\right)^{q/p}\\
&\lesssim\sum_{k=0}^\infty \left(\sum_{j=k}^\infty 
\sum_{w\in\Z^n} \sum_{l\in(j,w,k)} 2^{k(s-n/p)p}2^{(j-k)(\sigma-\frac np)p}|\eta_{j-k,w}^{k,l}|^p|\lambda_{k,l}|^p\right)^{q/p}\\
&=\sum_{k=0}^\infty 2^{k(s-n/p)q} \left(\sum_{j=0}^\infty 
\sum_{w\in\Z^n} \sum_{l\in(j+k,w,k)} 2^{j(\sigma-\frac np)p}|\eta_{j,w}^{k,l}|^p|\lambda_{k,l}|^p\right)^{q/p}\\
&=\sum_{k=0}^\infty 2^{k(s-n/p)q} \left(\sum_{l\in\Z^n} \sum_{j=0}^\infty 
\sum_{w\in\Z^n:l\in(j+k,w,k)} 2^{j(\sigma-\frac np)p}|\eta_{j,w}^{k,l}|^p|\lambda_{k,l}|^p\right)^{q/p}\\
&\lesssim\sum_{k=0}^\infty 2^{k(s-n/p)q} \left(\sum_{l\in\Z^n} |\lambda_{k,l}|^p\sum_{j=0}^\infty 
\sum_{w\in\Z^n} 2^{j(\sigma-\frac np)p}|\eta_{j,w}^{k,l}|^p\right)^{q/p}\\
&\le\sum_{k=0}^\infty 2^{k(s-n/p)q}\left(\sum_{l\in\Z^n} |\lambda_{k,l}|^p \right)^{q/p}=\|\lambda|b^s_{p,q}\|^q.
\end{align*}}
The proof of \eqref{eq:todo2} is finished. We again used \eqref{eq:have2} in the last inequality.
If $p$ and/or $q$ are equal to infinity, 
only notational changes are necessary.
\proofend

\remark{
Our results generalize \cite[Th.~2]{tri03} and \cite[Th.~2.3]{triw96},  where non-smooth atomic decompositions for spaces 
$\mathbf{B}^s_{p,p}(\rn)$ with $s>\max\left(n(1/p-1),0\right)$ can be found, to $\Bd(\rn)$ with no restrictions on the 
parameters. In particular, the case when $p\neq q$ is completely new.
}

Using the $\Lip$-atoms from Definition \ref{Lip-atom} and the embedding 
\[
\Lip(\rn) \hookrightarrow B^1_{\infty}(\rn),
\]
cf. \cite[p.89/90]{T-F1}, as a Corollary we now obtain  the following non-smooth atomic decomposition for Besov spaces with smoothness $0<s<1$.

\begin{corollary}\label{th:atom-dec-3}
 Let $0<p,q\leq\infty$, $0<s<1$, and $d>1$.  
Then $f\in L_p(\rn)$ belongs to $\Bd(\rn)$ if, and only if, it can be represented as
\beq\label{repr_1b}
f=\sum_{j=0}^\infty\sum_{m\in\zn}\lambda_{j,m}a_{j,m},
\eeq
where the $a_{j,m}$ are $\Lip$-atoms $(j\in\mathbb{N}_0)$ with
$\supp a_{j,m}\subset dQ_{j,m}$,  $j\in\mathbb{N}_0$, $m\in\zn,$
and $\lambda\in b^s_{p,q}$, convergence being in $L_p(\rn)$. Furthermore,
\beq\label{norm-inf1b}
\|f|\Bd(\rn)\|=\inf\|\lambda|b^s_{p,q}\|,
\eeq
where the infimum is taken over all admissible representations \eqref{repr_1b}.
\end{corollary}

\section{Spaces on  Lipschitz domains and their boundaries}
\label{sect-3}
\label{sec-3}


We call a one-to-one mapping
$
\Phi: \; \rn\mapsto \rn,
$ 
a \textit{Lipschitz diffeomorphism},  if the components $\Phi_k(x)$ of $\Phi(x)=(\Phi_1(x), \dots, \Phi_n(x))$ are Lipschitz functions on $\rn$ and 
\[
|\Phi(x)-\Phi(y)|\sim |x-y|, \quad x,y\in \rn, \; |x-y|\leq 1,
\] 
where the equivalence constants are independent of $x$ and $y$. Of course the inverse of $\Phi^{-1}$ is also a  Lipschitz diffeomorphism on $\rn$.

\begin{definition}\label{def-lip-1}
Let $\Omega$ be a bounded domain in $\rn$. Then $\Omega$ is said to be a  Lipschitz domain, if there exist $N$ open balls $K_1,\dots, K_N$ such that
$\ \ds  
\bigcup_{j=1}^N K_j\supset \Gamma\; \text{and}\; K_j\cap\Gamma\neq\emptyset\quad \text{if}\quad j=1,\dots,N,
\ $
with the following property: for every ball $K_j$ there are Lipschitz diffeomorphisms $\psi^{(j)}$ such that \\[0.4cm]
\begin{minipage}{0.45\textwidth}
\[
\psi^{(j)}:K_j\longrightarrow V_j,\quad j=1,\dots,N,
\]
\vspace{0.1cm}

where $V_j:=\psi^{(j)}(K_j)$ and 
\[
\psi^{(j)}(K_j\cap\Omega)\subset \rn_+,\qquad \psi^{(j)}(K_j\cap\Gamma)\subset \mathbb{R}^{n-1}.
\]
\end{minipage}\hfill
\begin{minipage}{0.5\textwidth}
\begin{psfrags}
 	\psfrag{K}{$K_j$}
	\psfrag{O}{$\Omega$}
	\psfrag{G}{$\Gamma=\partial\Omega$}
	\psfrag{a}{$\psi^{(j)}$}
	\psfrag{b}{$\left(\psi^{(j)}\right)^{-1}$}
	\psfrag{y}{$y$}
	\psfrag{v}{$y'$}
	\psfrag{V}{$V_j$}
	\psfrag{R}{$\mathbb{R}^{n-1}$}
	\psfrag{t}{$\psi^{(j)}(\Omega\cap K_j)$}
	\psfrag{s}{$\psi^{(j)}(\Gamma\cap K_j)$}
 	{\includegraphics[width=7.5cm]{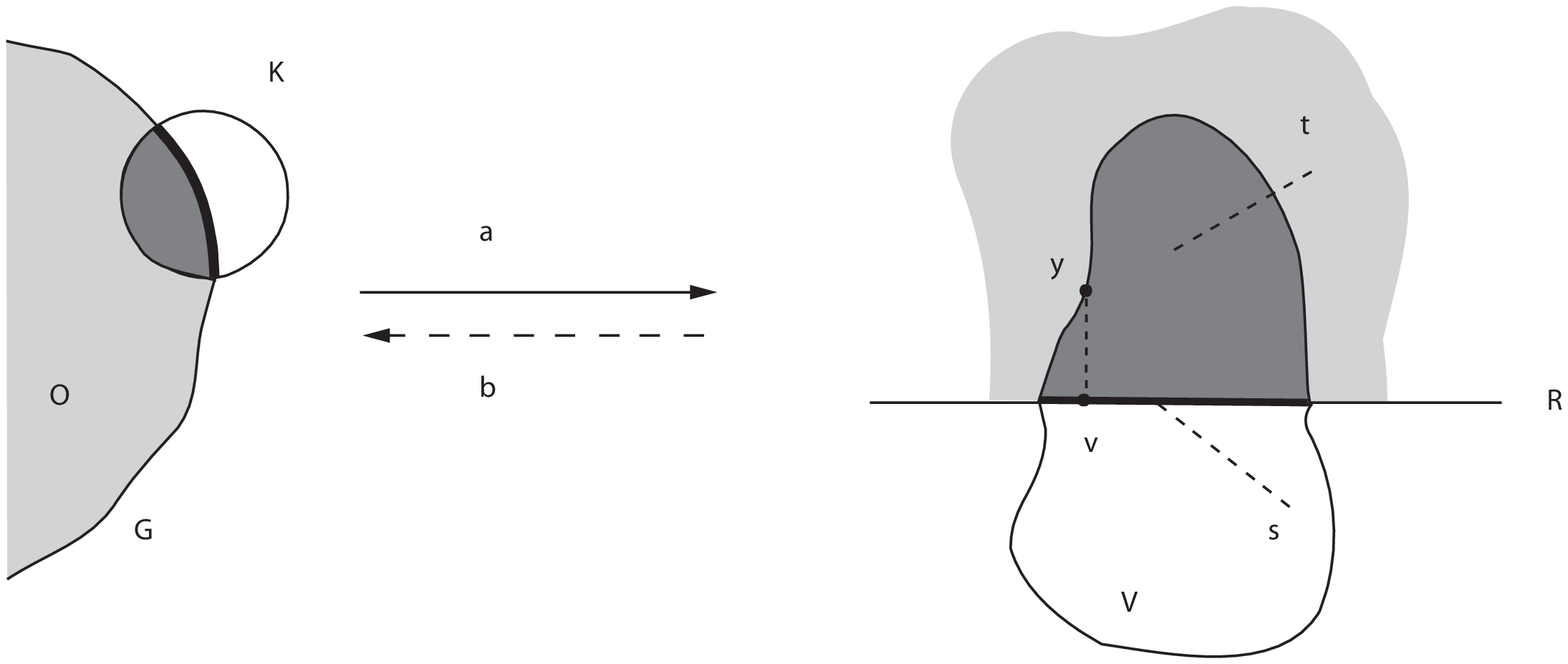}}
\end{psfrags}
\end{minipage}\\

\bigskip


\end{definition}

\remark{\label{def-lip-2}
The maps $\psi^{(j)}$ can be extended outside $K_j$ in such a way that the extended vector functions (denoted by $\psi^{(j)}$ as well) yield diffeomorphic mappings from $\rn$ onto itself (Lipschitz diffeomorphisms).\\
There are several equivalent definitions of Lipschitz domains in the literature. Our approach follows \cite{Dac04}. Another version as can be found in \cite{stein},  which defines first a
 {\em special $($unbounded$)$ Lipschitz domain} $\Omega$ in $\rn$ as simply the domain above the graph of a Lipschitz function $h: \mathbb{R}^{n-1}\longrightarrow \real$, i.e., 
\[
\Omega=\{(x',x_n): h(x')<x_n\}.
\]
Then a {\em bounded Lipschitz domain}  $\Omega$ in $\rn$ is defined as a bounded domain where the boundary $\Gamma=\partial\Omega$ can be covered by finitely many open balls $B_j$
 in $\rn$ with $j=1,\dots , J$, centered at $\Gamma$ such that 
 \[
B_j\cap \Omega=B_j\cap \Omega_j\qquad \text{for }j=1,\dots,J, 
 \]
 where $\Omega_j$ are rotations of suitable special Lipschitz domains in $\rn$. \\
 We shall occasionally use this alternative definition, in particular, since it usually suffices to consider special Lipschitz domains in our proofs (the related covering  involves only  finitely many balls), simplifying the notation considerably. \\
Consider a covering 
$\ 
\Omega\subset K_0\cup\left(\bigcup_{j=1}^N K_j\right),
\ $
where $K_0$ is an inner domain with $\overline{K}_0\subset\Omega$. Let $\{\varphi_j\}_{j=0}^N$ be a related {\em resolution of unity} of $\overline{\Omega}$, i.e., $\varphi_j$ are smooth nonnegative functions with support in $K_j$ additionally satisfying 
\beq\label{rou-2}
\sum_{j=0}^N \varphi_j(x)=1\quad \text{if }x\in\overline{\Omega}. 
\eeq
Obviously, the restriction of $\varphi_j$ to $\Gamma$ is a resolution of unity with respect to $\Gamma$. 
}

\bigskip

\subsection{Atomic decompositions for Besov spaces on boundaries}

The boundary $\partial\Omega=\Gamma$ of a bounded Lipschitz domain $\Omega$ will be furnished in the usual way with a surface measure $\ud \sigma$. The corresponding complex-valued Lebesgue spaces $L_p(\Gamma)$, $0<p\leq\infty$, are normed by
\[
\|g|L_p(\Gamma)\|=\left(\int_{\Gamma} |g(\gamma)|^p\ud \sigma(\gamma)\right)^{1/p}
\]
(with obvious modifications if $p=\infty$). We require the introduction of Besov spaces on $\Gamma$. We rely 
on the resolution of unity according to \eqref{rou-2} and the local Lipschitz diffeomorphisms  $\psi^{(j)}$ 
mapping $\Gamma_j=\Gamma\cap K_j$ onto $W_j=\psi^{(j)}(\Gamma_j)$, recall Definition \ref{def-lip-1}. We define
\[
g_j(y):=(\varphi_j f)\circ(\psi^{(j)})^{-1}(y),\qquad j=1,\dots,N,
\]
which restricted to $y=(y',0)\in W_j$,
\[
g_j(y')=(\varphi_j f)\circ(\psi^{(j)})^{-1}(y'),\qquad j=1,\dots,N, \quad f\in L_p(\Gamma),
\]
makes sense. This results in functions $g_j\in L_p(W_j)$ with compact supports in the $(n-1)$-dimensional Lipschitz domain $W_j$. 
We do not distinguish notationally between $g_j$ and $(\psi^{(j)})^{-1}$ as functions of $(y',0)$ and of $y'$. \\


Our constructions enable us to transport  Besov spaces naturally from $\real^{n-1}$ to the  boundary $\Gamma$ of a (bounded) 
Lipschitz domain via pull-back and a partition of unity.\\

\begin{definition}\label{def-Bq-bd}
 Let $n\geq 2$, and let $\Omega$ be a bounded Lipschitz domain in $\rn$ with boundary $\Gamma$, and $\varphi_j$, $\psi^{(j)}$, $W_j$ be as above. Assume $0<s<1$ and $0<p,q\leq\infty$. Then we introduce
\[
\Bd(\Gamma)=\{f\in L_p(\Gamma):g_j\in \Bd(W_j), \; j=1,\dots,N\},
\]
equipped with the quasi-norm
$\ \ds 
\|f|\Bd(\Gamma)\|:=\sum_{j=1}^N\|g_j|\Bd(W_j)\|.
\ $
\end{definition}

\remark{\label{rem-bd-boundary} The spaces $\Bd(\Gamma)$ turn out to be independent of the particular choice of the resolution of unity $\{\varphi_j\}_{j=1}^N$ and the local diffeomorphisms $\psi^{(j)}$ (the proof is similar to the proof of \cite[Prop.~3.2.3(ii)]{T-F1}, making use of Propositions \ref{prop-mult-diffeo} and \ref{equ-norm-dom} below). 
We furnish $\Bd(W_j)$ with the intrinsic $(n-1)$-dimensional norms according to Definition \ref{intr_char}. 
Note that we could furthermore replace $W_j$ in the definition of the norm above by $\mathbb{R}^{n-1}$ if we extend 
$g_j$ outside $W_j$ with zero, i.e., 
\beq\label{Equiv}
\|f|\Bd(\Gamma)\|\sim\sum_{j=1}^N\|g_j|\Bd(\mathbb{R}^{n-1})\|.
\eeq
In particular,  the equivalence \eqref{Equiv} yields that  characterizations for B-spaces defined on $\real^{n-1}$ can be generalized to B-spaces defined on $\Gamma$. 
This will be done in Theorem \ref{th:at-bd-1} for non-smooth atomic decompositions and is very likely to work as well for characterizations in terms of differences. 
}

\paragraph{Atomic decompositions for $\mathbf{B}^s_{p,q}(\Gamma)$}

Similar to the non-smooth atomic decompositions constructed in Section \ref{sect-2} we now establish corresponding atomic decompositions for Besov spaces defined on Lipschitz boundaries. They will be very useful when investigating traces on Lipschitz domains in Section \ref{sect-3}\\

The relevant sequence spaces and Lipschitz-atoms on the boundary $\Gamma$ we shall define next are closely related to the sequence spaces $b^s_{p,q}(\Omega)$ and $\Lip$-atoms used for the non-smooth atomic decompositions as used in  Corollary \ref{th:atom-dec-3}. 


\begin{definition}\label{def-seq-gamma}
Let $0<p,q\leq\infty$, $s\in\real$. Furthermore, let $\Gamma$ be the boundary of a bounded Lipschitz domain $\Omega\subset\rn$,  and $\lambda=\{\lambda_{j,m}\in\mathbb{C}:j\in\mathbb{N}_0, m\in\zn\}$. 
Then\index{sequence spaces!of type $b^s_{p,q}(\Gamma)$}
\[
b^s_{p,q}(\Gamma)=\left\{\lambda: \|\lambda|b^s_{p,q}(\Gamma)\|=\left(\sum_{j=0}^{\infty}2^{j(s-\frac{n-1}{p})q}\left({\sum_{m\in\zn}}^{\Gamma,j}|\lambda_{j,m}|^p\right)^{q/p}\right)^{1/q}<\infty\right\}
\]
$($with the usual modification if $p=\infty$ and/or $q=\infty$$)$.  
\end{definition}

\begin{definition}\label{bd-lip-atom}
Let $j\in \nat_0$, $m\in\zn$, $d>1$, and let $\Gamma$ be the boundary of a bounded Lipschitz domain $\Omega\subset \rn$. Put $Q_{j,m}^{\Gamma}:=dQ_{j,m}\cap \Gamma\neq \emptyset$. A function $a\in\Lip(\Gamma)$ is a $\Lip^{\Gamma}$-atom, if 
\[
\supp a\subset  Q_{j,m}^{\Gamma},  \qquad d>1,
\]
\beq\label{lip-gamma-cond}
\|a|L_{\infty}(\Gamma)\|\leq 1 \quad \text{and }\quad \sup_{x,y\in\Gamma,\atop x\neq y}\frac{|a(x)-a(y)|}{|x-y|}\leq 2^j.
\eeq
\end{definition}

\remark{
Note that if we put $2^j\Gamma:=\{2^jx:x\in\Gamma\}$, we can state \eqref{lip-gamma-cond} like  $\|a(2^{-j}\cdot)|\Lip(2^{j}\Gamma)\|\le 1.$ 
}

The theorem below provides atomic decompositions for the spaces $\mathbf{B}^s_{p,q}(\Gamma)$. \\

\begin{theorem}\label{th:at-bd-1}
Let $\Omega\subset \rn$ be a bounded Lipschitz domain and let $0<s<1$, $0<p,q \leq \infty$. Then $f\in L_p(\Gamma)$ belongs to $\mathbf{B}^s_{p,q}(\Gamma)$ if, and only if, 
\[
f={\sum_{j,m}}\lambda_{j,m}a_{j,m},
\]
where $a_{j,m}$ are $\Lip^{\Gamma}$-atoms with
$
\supp a_{j,m}\subset Q_{j,m}^{\Gamma}
$
and $\lambda\in b^s_{p,q}(\Gamma)$, convergence being in $L_p(\Gamma)$. Furthermore,
\[
\|f|\mathbf{B}^s_{p,q}(\Gamma)\|=\inf \|\lambda|b^s_{p,q}(\Gamma)\|,
\]
where the infimum is taken over all possible representations. 
\end{theorem}

\proofstart 

\underline{Step 1:} {Fix $f\in {\bf B}^s_{p,q}(\Gamma)$. For simplicity, we suppose that 
$\supp f\subset \{x\in\Gamma:\varphi_l(x)=1\}$ for some $l\in\{1,2,\dots,N\}$. If this is not the case the arguments have to 
be slightly modified to incorporate the decomposition of unity \eqref{rou-2}. To simplify the notation we write $\varphi$
instead of $\varphi_l$ and $\psi$ instead of $\psi^{(l)}$. }Then we obtain
$$
\|f|\mathbf{B}^s_{p,q}(\Gamma)\|=\|f\circ {\psi}^{-1}|\mathbf{B}^s_{p,q}(\real^{n-1})\|.
$$
We use Corollary \ref{th:atom-dec-3} with $n$ replaced by $n-1$ to obtain an optimal atomic decomposition
\begin{equation}\label{eq:prdec1}
f\circ \psi^{-1}=\sum_{j,m}\lambda_{j,m}a_{j,m} \quad \text{where} \quad \|f\circ\psi^{-1}|\mathbf{B}^s_{p,q}(\R^{n-1})\|\sim \|\lambda|b^s_{p,q}(\real^{n-1})\|. 
\end{equation}

For $j\in\N_0$ and $m\in\Z^{n-1}$ fixed, we consider the function $a_{j,m}(\psi(x))$. Due to the Lipschitz properties of $\psi$,
this function is supported in $Q^\Gamma_{j,l}$ for some $l\in\Z^{n}$ and we denote it by $a^\Gamma_{j,l}(x).$
Furthermore, we set $\lambda'_{j,l}=\lambda_{j,m}$. This leads to the decomposition
\begin{equation}\label{eq:prdec2}
f=\sum_{j,l}\lambda'_{j,l}a^\Gamma_{j,l}.
\end{equation}
It is straightforward to verify that $a^\Gamma_{j,l}$ are $\Lip^{\Gamma}$-atoms since 
$\ 
\|a_{j,l}^{\Gamma}|L_{\infty}(\Gamma)\|\lesssim  \|a_{j,m}|L_{\infty}(W_l)\|\lesssim 1
\ $ 
and 
\[
\frac{|a_{j,l}^{\Gamma}(x)-a_{j,l}^{\Gamma}(y)|}{|x-y|}
=  \frac{|a_{j,m}(x')-a_{j,m}(y')|}{|{\psi}^{-1}(x')-{\psi}^{-1}(y')|}\\
\sim   \frac{|a_{j,m}(x')-a_{j,m}(y')|}{|x'-y'|}
\lesssim  2^{j}, \qquad x,y\in\Gamma. 
\]

Furthermore,  we have the estimate
$$
\|f|\mathbf{B}^s_{p,q}(\Gamma)\|=\|f\circ {\psi}^{-1}|\mathbf{B}^s_{p,q}(\real^{n-1})\|
\sim \|\lambda|b^s_{p,q}(\real^{n-1})\|=\|\lambda'|b^s_{p,q}(\Gamma)\|.
$$

\underline{Step 2:}

The proof of the opposite direction follows along the same lines. If $f$ on $\Gamma$ is given by
$$
f=\sum_{j,l}\lambda'_{j,l}a^\Gamma_{j,l},
$$
then
$\ \ds  
f\circ\psi^{-1}=\sum_{j,m}\lambda_{j,m}a_{j,m},
\ $
where $a_{j,m}(x)=a^\Gamma_{j,l}(\psi^{-1}(x))$ and $\lambda_{j,m}=\lambda'_{j,l}$ for suitable $m\in\Z^{n-1}$.
Again it follows that $a_{j,m}$ are ${\rm Lip}$-atoms on $\R^{n-1}$ and
$$
\|f|\mathbf{B}^s_{p,q}(\Gamma)\|=\|f\circ {\psi}^{-1}|\mathbf{B}^s_{p,q}(\real^{n-1})\|
\lesssim \|\lambda|b^s_{p,q}(\real^{n-1})\|=\|\lambda'|b^s_{p,q}(\Gamma)\|.
$$

\underline{Step 3:}\quad The convergence in $L_p(\Gamma)$ of the representation 
$\ \ds 
f={\sum_{j,m}}^{j,\Gamma}\lambda_{j,m} a_{j,m}^{\Gamma},
\ $ 
follows for $p\le 1$ by
\begin{align}
\left\|{\sum_{j,m}}^{j,\Gamma}\lambda_{j,m}a^{\Gamma}_{j,m}|L_p(\Gamma)\right\|^p
&\leq {\sum_{j,m}}^{j,\Gamma}|\lambda_{j,m}|^p\|a^{\Gamma}_{j,m}|L_p(\Gamma)\|^p\notag\\
&\lesssim \sum_{j}2^{-j(n-1)}{\sum_{m}}^{j,\Gamma}|\lambda_{j,m}|^p = \|\lambda|b^{0}_{p,p}(\Gamma)\|^p\lesssim \|\lambda|b^{s}_{p,q}(\Gamma)\|^p\label{help-e1}
\end{align}
and using
\begin{align}
\left\|{\sum_{j,m}}^{j,\Gamma}\lambda_{j,m}a^{\Gamma}_{j,m}|L_p(\Gamma)\right\|&\le
\sum_j\left\|{\sum_{m}}^{j,\Gamma}\lambda_{j,m}a^{\Gamma}_{j,m}|L_p(\Gamma)\right\|
\lesssim \sum_j 2^{-j(n-1)/p}\left({\sum_{m}}^{j,\Gamma}|\lambda_{j,m}|^p\right)^{1/p}\notag\\
&=\|\lambda|b^{0}_{p,1}(\Gamma)\|\lesssim \|\lambda|b^{s}_{p,q}(\Gamma)\|\label{help-e2}
\end{align}
for $p>1$.
\proofend

\subsection{Interpolation results}

Interpolation results for  $\Bd(\rn)$ as obtained in \cite[Cor.~6.2, 6.3]{dVP} carry over to the spaces $\Bd(\Gamma)$,
which follows immediately from their definition and properties of real  interpolation.

\begin{theorem}\label{th:interpol-B-gamma}
Let $\Omega$ be a bounded Lipschitz domain with boundary $\Gamma$. 

\bit
\item[(i)] Let $0<p,q,q_0,q_1\leq \infty$, $s_0\neq s_1$, and $0<s_i<1$.  Then 
\[
\left(\mathbf{B}^{s_0}_{p,q_0}(\Gamma), \mathbf{B}^{s_1}_{p,q_1}(\Gamma)\right)_{\theta,q}=\Bd(\Gamma),
\]
where $0<\theta<1$ and $s=(1-\theta)s_0+\theta s_1$.

\item[(ii)] Let $0<p_i,q_i\leq \infty$, $s_0\neq s_1$ and $0<s_i<1$. Then for each $0<\theta<1$, $s=(1-\theta)s_0+\theta s_1$, 
$ 
\frac 1p=\frac{1-\theta}{p_0}+\frac{\theta}{p_1},$  \text{and for } $\frac 1q=\frac{1-\theta}{q_0}+\frac{\theta}{q_1}
$
we have
\[
\left(\mathbf{B}^{s_0}_{p_0,q_0}(\Gamma), \mathbf{B}^{s_1}_{p_1,q_1}(\Gamma)\right)_{\theta,q}=\Bd(\Gamma),
\]
provided $p=q$. 
\eit
\end{theorem}

\proofstart By definition of the spaces $\Bd(\Gamma)$ we can construct a well-defined and bounded linear operator 
\[
E:\Bd(\Gamma)\longrightarrow \oplus_{1\leq j\leq N} \Bd(\real^{n-1}),
\]
\[
(Ef)_j:=(\varphi_j f)\circ {\psi^{(j)}}^{-1}\quad \text{on}\; \real^{n-1}, \quad 1\leq j\leq N,
\]
which has a bounded and linear left inverse given by
\[
R:\oplus_{1\leq j\leq N}\Bd(\real^{n-1})\longrightarrow \Bd(\Gamma)
\]
\[
R\left((g_j)_{1\leq j\leq N}\right):=\sum_{j=1}^N \Psi_j \left(g_j\circ \psi_j\right)\qquad \text{on}\quad \Gamma,
\]
where $\Psi_j\in C_0^{\infty}(\rn)$, $\supp \Psi_j\subseteq K_j$, $\Psi\equiv 1$ in a neighborhood of $\supp \varphi_j$. \\

Straightforward calculation shows for $f\in \Bd(\Gamma)$
\[
(R\circ E)f=R(Ef)=R\left(\left((\varphi_jf)\circ {\psi^{(j)}}^{-1}\right)_{1\leq j\leq N}\right)=\sum_{j=1}^N \Psi_j\varphi_j f=\sum_{j=1}^N \varphi_j f=f,
\]
i.e., 
\[
R\circ E=I, \quad \text{the identity operator on }\Bd(\Gamma).
\]
One arrives at a standard situation in interpolation theory. Hence, by the method of retraction-coretraction, cf. \cite[Sect.~1.2.4, 1.17.1]{T-I}, the results for $\Bd(\real^{n-1})$ carry over to the spaces $\Bd(\Gamma)$. 
Therefore, (i) and (ii) are a consequence of \cite[Cor.~6.2, 6.3]{dVP}.
\proofend

Furthermore, we briefly show that the interpolation results for Besov spaces $\Bd(\rn)$ also hold for spaces on domains $\Bd(\Omega)$. This is not automatically clear in our context since the extension operator 
\[
\Ex:\Bd(\Omega)\longrightarrow \Bd(\rn)
\] 
constructed in \cite{dVS93} is not linear.  The situation is different for spaces $\B(\Omega)$.  Here Rychkov's (linear) extension operator, cf. \cite{ryc00}, automatically yields interpolation results for B-spaces on domains. 

\begin{theorem}\label{th:interpol-B-omega}
Let $\Omega$ be a bounded Lipschitz domain. 

\bit
\item[(i)] Let $0<p,q,q_0,q_1\leq \infty$, $s_0\neq s_1$, and $0<s_i<1$. Then 
\[
\left(\mathbf{B}^{s_0}_{p,q_0}(\Omega), \mathbf{B}^{s_1}_{p,q_1}(\Omega)\right)_{\theta,q}=\Bd(\Omega),
\]
where $0<\theta<1$ and $s=(1-\theta)s_0+\theta s_1$.

\item[(ii)] Let $0<p_i,q_i\leq \infty$, $s_0\neq s_1$ and $0<s_i<1$. Then for each $0<\theta<1$, $s=(1-\theta)s_0+\theta s_1$, 
$ 
\frac 1p=\frac{1-\theta}{p_0}+\frac{\theta}{p_1},$  \text{and for } $\frac 1q=\frac{1-\theta}{q_0}+\frac{\theta}{q_1}
$
we have
\[
\left(\mathbf{B}^{s_0}_{p_0,q_0}(\Omega), \mathbf{B}^{s_1}_{p_1,q_1}(\Omega)\right)_{\theta,q}=\Bd(\Omega),
\]
provided $p=q$. 
\eit
\end{theorem}

\proofstart  In spite of our remarks before the theorem, we can nevertheless use the extension operator 
\[
\Ex:\Bd(\Omega)\longrightarrow \Bd(\rn)
\]
constructed in \cite{dVS93} to show that interpolation results for spaces $\Bd(\rn)$ carry over to spaces $\Bd(\Omega)$. Let $X_i(\Omega):=\mathbf{B}^{s_i}_{p_i,q_i}(\Omega)$. By the explanations given in \cite[p.~859]{dVS93} we have the estimate
\beq\label{est-interpol}
K(f,t,X_0(\Omega), X_1(\Omega))\sim K(\Ex f, t, X_0(\rn), X_1(\rn))
\eeq
although the operator $\Ex$  is not linear. Let $\mathbf{B}^{\theta}(\Omega):=\left(\mathbf{B}^{s_0}_{p_0,q_0}(\Omega), \mathbf{B}^{s_1}_{p_1,q_1}(\Omega)\right)_{\theta,q}$ with the given restrictions on the parameters given in (i) and (ii), respectively. We have to prove that 
\[
\mathbf{B}^{\theta}(\Omega)=\Bd(\Omega),
\] 
but this follows immediately from \cite[Cor.~6.2,6.3]{dVP} using \eqref{est-interpol}, since
\[
\|f|\mathbf{B}^{\theta}(\Omega)\|\sim\|\Ex f|\mathbf{B}^{\theta}(\rn)\|\sim \|\Ex f|\Bd(\rn)\|\sim \|f|\Bd(\Omega)\|.
\] 
\proofend

\subsection{Properties of Besov spaces on Lipschitz domains}

The non-smooth atomic decomposition enables us to generalize \cite[Prop.~2.5]{sch08c} and obtain new results concerning 
diffeomorphisms and pointwise multipliers in $\Bd(\rn)$ in the following way. For related matters we also refer to 
\cite[Th.~3.3.3]{May05}.

\begin{proposition}\label{prop-mult-diffeo}
Let $0<p,q\leq\infty$, $0<s<1$ and $\sigma>s$.
\begin{enumerate}
\item[\rm (i)] $($Diffeomorphisms$)$\\ 
Let $\psi$ be a Lipschitz diffeomorphism. Then 
$f\longrightarrow f\circ \psi$
is a linear and bounded operator from $\Bd(\rn)$ onto itself.
\item[\rm (ii)] $($Pointwise multipliers$)$\\ 
Let $h\in {\mathcal C}^{\sigma}(\rn)$. Then 
$f\longrightarrow hf$ 
is a linear and bounded operator from $\Bd(\rn)$ into itself.
\end{enumerate}
\end{proposition}

\proofstart
Concerning (i), we make use of the atomic decomposition as in \eqref{repr_1b} with the Lip-atoms from Definition \ref{Lip-atom}. 
Then we have 
\[
f\circ \psi=\sum_{j=0}^{\infty}\sum_{m\in\zn}\lambda_{j,m}a_{j,m}\circ \psi
\]
and $a\circ \psi$ is a $\Lip$-atom based on a new cube, and multiplied with a constant depending on $\psi$,  since
\[
|(a_{j,m}\circ \psi)(x)-(a_{j,m}\circ \psi)(y)|\leq 2^j |\psi(x)-\psi(y)|\lesssim 2^j |x-y|
\]

To prove (ii) we argue as follows. First, we may suppose that $0<s<\sigma<1$. Furthermore, we choose a real parameter $\sigma'$ with
$s<\sigma'<\sigma$. We take the smooth atomic decomposition \eqref{repr_01'} with $K$-atoms $a_{j,m}$, where $K=1$.
Multiplied with $h\in {\mathcal C}^{\sigma}$, it gives a new (non-smooth) atomic decomposition of $hf$. Its convergence in $L_p(\R^n)$
follows from the convergence of \eqref{repr_01'} in $L_p(\R^n)$ and the boundedness of $h$.\\
It remains to verify, that $ha_{j,m}$ are  non-smooth $(\sigma',p)$-atoms. The support property follows immediately
from the support property of $a_{j,m}$.
We use the bounded support of $(ha_{j,m})(2^{-j}\cdot)$ and the multiplier assertion for ${\mathbf B}^\sigma_\infty(\R^n)$ as presented 
in \cite[Section 4.6.1,Theorem 2]{RS96} to get
\begin{align*}
\|(h a_{j,m})(2^{-j}\cdot)|{\mathbf B}^{\sigma'}_{p}(\R^n)\|&\le \|(h a_{j,m})(2^{-j}\cdot)|{\mathbf B}^{\sigma}_{\infty}(\R^n)\|\\
&=\|h(2^{-j}\cdot)\cdot a_{j,m}(2^{-j}\cdot)|{\mathbf B}^{\sigma}_{\infty}(\R^n)\|\\
&\lesssim \|h(2^{-j}\cdot)|{\mathbf B}^{\sigma}_{\infty}(\R^n)\|\cdot\|a_{j,m}(2^{-j}\cdot)|{\mathbf B}^{\sigma}_{\infty}(\R^n)\|.
\end{align*}
The last product is bounded by a constant due to the inequality 
$$
\|h(2^{-j}\cdot)|{\mathbf B}^{\sigma}_{\infty}(\R^n)\|\lesssim \|h|{\mathbf B}^{\sigma}_{\infty}(\R^n)\|,\quad j\in\N_0,
$$
which may be verified directly (or found in \cite[Section 1.7]{Bo83} or \cite[Section 2.3.1]{ET96}), 
combined with the fact that $a_{j,m}$ are $K$-atoms for $K=1.$
\proofend

Furthermore, we establish an equivalent quasi-norm for $\Bd(\Omega)$.

\begin{proposition}\label{equ-norm-dom}
 Let $0<p,q\leq\infty$, $0<s<1$, and $\Omega$ be a bounded Lipschitz domain. Then
\beq\label{equiv-norm-1}
\|\varphi_0 f|\Bd(\rn)\|+\sum_{j=1}^N \|(\varphi_j f)(\psi^{(j)}(\cdot))^{-1}|\Bd(\rn_+)\|
\eeq
is an equivalent quasi-norm in $\Bd(\Omega)$.
\end{proposition}

\begin{proof}
 Let $\Omega_1$ be a bounded domain with
\[
{\overline{\Omega}_1}\subset\left\{x\in\rn: \sum_{j=0}^N\varphi_j(x)=1\right\}
\]
and $\overline{\Omega}\subset \Omega_1$. Let $f\in\Bd(\Omega)$. If we restrict the infimum in \eqref{extr_char} to $g\in\Bd(\rn)$ with
\beq\label{cond-g}
g\big|_\Omega=f\qquad \text{and}\qquad \supp g\subset \Omega_1,
\eeq
then we obtain a new equivalent quasi-norm in $\Bd(\Omega)$. This follows from Proposition \ref{prop-mult-diffeo}(ii) if one multiplies an arbitrary element $g\in\Bd(\rn)$ with a fixed infinitely differentiable function $\varkappa(x)$ with 
\[
\varkappa(x)=1\quad\text{if}\quad x\in\Omega \qquad \text{and}\qquad \supp\varkappa\subset \Omega_1.
\]
For elements $g\in\Bd(\rn)$ with \eqref{cond-g}, 
\[
\sum_{k=0}^N\|\varphi_kg|\Bd(\rn)\|
\] 
is an equivalent quasi-norm. This is also a consequence of Proposition \ref{prop-mult-diffeo}(ii). Applying part (i) of that proposition to $g(x)\rightarrow g(\psi^{(j)}(x))$, we see that
\[
\|\varphi_0 g|\Bd(\rn)\|+\sum_{k=1}^N \|(\varphi_k g)(\psi^{(k)}(\cdot))^{-1}|\Bd(\rn)\|
\]
is an equivalent quasi-norm for all $g\in\Bd(\rn)$ with \eqref{cond-g}. But the infimum over all admissible $g$ with \eqref{cond-g} yields \eqref{equiv-norm-1}.
\end{proof}

\section{Trace results on Lipschitz domains}

Now we can look for traces of $f\in \Bd(\Omega)$ on the boundary $\Gamma$.
We briefly explain our understanding of the trace operator since
when dealing with $L_p(\R^n)$ functions the pointwise trace has no obvious meaning. \\ 
Let $Y(\Gamma)$ denote one of the spaces $\mathbf{B}^\sigma_{u,v}(\Gamma)$ or $L_u(\Gamma)$. 
Since $\ensuremath{\mathcal{S}(\Omega)}$ is dense in $\Bd(\Omega)$ for $0<p,q<\infty$ (both spaces can be interpreted 
as restrictions of their counterparts defined on $\rn$), one asks first whether there is a constant $c>0$ such that
\beq\label{trace-02}
\|\tr \varphi|Y(\Gamma)\|\leq c\|\varphi|\Bd(\Omega)\| \quad \text{for all }\varphi\in \ensuremath{\mathcal{S}(\Omega)},
\eeq
where $\ensuremath{\mathcal{S}(\Omega)}$ stands for the restriction of the Schwartz space $\ensuremath{\mathcal{S}(\rn)}$ to a domain 
$\Omega$. If this is the case, then one defines $\tr f\in Y(\Gamma)$ for $f\in \mathbf{B}^s_{p,q}(\Omega)$ 
by completion and obtains
\[
\|\tr f|Y(\Gamma)\|\leq c\|f|\Bd(\Omega)\|, \quad f\in \mathbf{B}^s_{p,q}(\Omega),
\]
for the linear and bounded trace operator\index{trace operators!on the boundary}
\[
\tr: \mathbf{B}^s_{p,q}(\Omega)\hookrightarrow Y(\Gamma).
\]

\remark{\label{trace_infty}
We can extend \eqref{trace-02} to spaces $\Bd(\Omega)$ with $p=\infty$ and/or $q=\infty$ by using embeddings for 
B- and F-spaces from \cite{HS08, sch08a}. The results stated there can be generalized to domains $\Omega$, 
since the spaces $\Bd(\Omega)$ are defined by restriction of the corresponding spaces on $\rn$, cf. Remark \ref{extr_char}.\\
If $p=\infty$, we have that $\mathbf{B}^s_{\infty,q}(\Omega)$ with $s>0$ is embedded in the space of continuous functions and $\tr$
makes sense pointwise.
If $q=\infty$,
\[
\mathbf{B}^s_{p,\infty}(\Omega)\hookrightarrow \mathbf{B}^{s-\varepsilon}_{p,1}(\Omega) \qquad \text{for any}\quad \varepsilon>0.
\]
Let $s>\frac 1p$ and $\varepsilon>0$ be small enough such that one has
\[
s>s-\varepsilon>\frac 1p.
\]
Since by \cite[Rem.~13]{tri08} traces are independent of the source spaces and of the target spaces one can now define $\tr$ for 
$\mathbf{B}^s_{p,\infty}(\Omega)$ by restriction of $\tr$  for $\mathbf{B}^{s-\varepsilon}_{p,1}(\Omega)$ to 
$\mathbf{B}^s_{p,\infty}(\Omega)$. Hence \eqref{trace-02} is always meaningful.
}

\subsection{Boundedness of the trace operator}

Now we are able to state and prove our first main theorem concerning traces of Besov spaces on Lipschitz domains.

\begin{theorem}\label{B-trace-1}
Let $n\geq 2$, $0<p,q\leq\infty$, $0<s<1$, and let $\Omega$ be a bounded Lipschitz domain in $\rn$ with boundary $\Gamma$. Then
the operator
\beq\label{trace-th1}
\tr: \mathbf{B}^{s+\frac 1p}_{p,q}(\Omega)\longrightarrow\mathbf{B}^{s}_{p,q}(\Gamma)
\eeq
is linear and bounded.
\end{theorem}

\begin{proof}
The linearity of the operator follows directly from its definition as discussed above.
To prove the boundedness, we take an optimal representation of a smooth function $f\in \mathbf{B}^{s+\frac 1p}_{p,q}(\Omega)$ as described in \eqref{repr_01'}, i.e.,
\beq\label{sums-sense'}
f=\sum_{j=0}^{\infty}{\sum_{m\in \zn}}^{j,\Omega}\lambda_{j,m}a_{j,m} \quad \text{with}\quad \|f|\mathbf{B}^{s+\frac 1p}_{p,q}(\Omega)\|\sim
\|\lambda|b^{s+\frac 1p}_{p,q}(\Omega)\|.
\eeq
We put
\beq\label{sums-sense}
\tr f:=\left({\sum_{j,m}}^{j,\Omega}\lambda_{j,m}a_{j,m}\right)\Bigg|_{\Gamma}=
{\sum_{j,m}}^{j,\Gamma}\lambda_{j,m}a_{j,m}\Big|_{\Gamma}={\sum_{j,m}}^{j,\Gamma}\lambda_{j,m}a_{j,m}^{\Gamma}.
\eeq
The proof follows by Theorem \ref{th:at-bd-1} and the following four facts:
\begin{enumerate}
\item[(i)] $a_{j,m}^{\Gamma}$ are $\Lip^{\Gamma}$-atoms,
\item[(ii)] $\|\lambda|b^{s}_{p,q}(\Gamma)\|\lesssim \|\lambda|b^{s+\frac 1p}_{p,q}(\Omega)\|$,
\item[(iii)] the decomposition \eqref{sums-sense} converges in $L_p(\Gamma)$,
\item[(iv)] the trace operator $\tr$ coincides with the trace operator discussed above.
\end{enumerate}

To prove the first point, we observe that
\[
\supp a_{j,m}^{\Gamma}\subseteq \supp a_{j,m}\cap \Gamma \subseteq Q_{j,m}^{\Gamma}.
\]
Furthermore, we have $\ \|a_{j,m}^{\Gamma}|L_{\infty}(\Gamma)\|\leq \|a_{j,m}|L_{\infty}(dQ_{j,m})\|\leq c\ $ and
$$
\sup_{ x,y\in Q^{\Gamma}_{j,m}\atop x\neq y}\frac{a_{j,m}^{\Gamma}(x)-a_{j,m}^{\Gamma}(y)}{|x-y|}\leq \sup_{x,y\in dQ_{j,m}\atop x\neq y}
\frac{a_{j,m}(x)-a_{j,m}(y)}{|x-y|}\lesssim 2^{j}.
$$
The proof of the second point follows directly by
\begin{align*}
\|\lambda|b^{s}_{p,q}(\Gamma)\|&=\left(\sum_{j}2^{j(s-\frac{n-1}{p})q}\left({\sum_{m}}^{j,\Gamma}|\lambda_{j,m}|^p\right)^{q/p}\right)^{1/p
}\\
&\le\left(\sum_{j}2^{j\left[(s+\frac 1p)-\frac{n}{p}\right]q}\left({\sum_{m}}^{j,\Omega}|\lambda_{j,m}|^p\right)^{q/p}\right)^{1/p}
=\|\lambda|b^{s+\frac 1p}_{p,q}(\Omega)\|.
\end{align*}
The proof of the third point follows  in the same way as the proof in Step 3 of Theorem \ref{th:at-bd-1}. 

The proof of (iv) is based on the fact that for $f\in{\mathcal S}(\Omega)$ there is an optimal atomic decomposition \eqref{sums-sense'}
which converges also pointwise. This may be observed by a detailed inspection of \cite{H-N}. Therefore also the series \eqref{sums-sense}
converges pointwise and the trace operator $\tr$ may be understood in the pointwise sense for smooth $f$.
\end{proof}

\subsection{Extension of atoms}

In order to compute the exact trace space we still need to construct an extension operator
\[
\mathit{Ext}: \mathbf{B}^{s}_{p,q}(\Gamma)\longrightarrow \mathbf{B}^{s+\frac 1p}_{p,q}(\Omega)
\]
and show its boundedness. The main problem will be to show that we can extend the $\Lip^{\Gamma}$-atoms from the source
spaces in a nice  way to obtain suitable atoms for the target spaces. We start with a simple variant of the
Gagliardo-Nirenberg inequality, cf. \cite[Chapter 5]{P76}.

\begin{lemma}\label{lem1}
Let $0<s_0,s_1<\infty$, $0<p_0,p_1,q_0,q_1\le \infty$ and $0<\theta<1$. Put
\begin{equation}
s=(1-\theta)s_0+\theta s_1, \quad \frac{1}{p}=\frac{1-\theta}{p_0}+\frac{\theta}{p_1}, \quad
\frac{1}{q}=\frac{1-\theta}{q_0}+\frac{\theta}{q_1}.
\end{equation}
Then
\begin{equation}
\|f|\mathbf {B}^s_{p,q}(\Omega)\|\lesssim
\|f|\mathbf {B}^{s_0}_{p_0,q_0}(\Omega)\|^{1-\theta}\cdot \|f|\mathbf {B}^{s_1}_{p_1,q_1}(\Omega)\|^\theta
\end{equation}
for all $f\in \mathbf{B}^{s_0}_{p_0,q_0}(\Omega) \cap \mathbf{B}^{s_1}_{p_1,q_1}(\Omega)$.
\end{lemma}
\begin{proof}
The straightforward proof uses the characterization of ${B}$-spaces through differences and H\"older's inequality.
\end{proof}

Our approach is based on the classical Whitney decomposition of $\R^n\setminus \Gamma$ and the corresponding
decomposition of unity. We summarize the most important properties of this method in the next Lemma and 
refer to \cite[pp.167-170]{stein} and \cite[pp.21-26]{JW84} for details and proofs.

\begin{lemma}\label{lem:whitney}
1. Let $\Gamma\subset\R^n$ be a closed set. Then there exists a collection of cubes $\{Q_i\}_{i\in\N}$, such that
\begin{enumerate}
\item[(i)] $\R^n\setminus \Gamma=\bigcup_{i}Q_i$.
\item[(ii)] The interiors of the cubes are mutually disjoint.
\item[(iii)] The inequality
$$
{\rm diam}\ Q_i\le {\rm dist}\,(Q_i,\Gamma)\le 4\,{\rm diam}\ Q_i
$$
holds for every cube $Q_i$. Here ${\rm diam}\ Q_i$ is the diameter of $Q_i$ and ${\rm dist}\, (Q_i,\Gamma)$
is its distance from $\Gamma.$
\item[(iv)] Each point of $\R^n\setminus\Gamma$ is contained in at most $N_0$ cubes $6/5\cdot Q_i$, where $N_0$
depends only on $n$.
\item[(v)] If $\Gamma$ is the boundary of a Lipschitz domain then 
there is a number $\gamma>0$, which depends only on $n$, such that $\sigma(\gamma Q_i\cap \Gamma)>0$
for all $i\in\N.$
\end{enumerate}
2. The are $C^\infty$-functions $\{\psi_i\}_{i\in\N}$ such that
\begin{enumerate}
\item[(i)] $\sum_i \psi_i(x)=1$ for every $x\in\R^n\setminus \Gamma$.
\item[(ii)] ${\rm supp}\,\psi_i\subset 6/5\cdot Q_i$.
\item[(iii)] For every $\alpha\in\N_0^n$ there is a constant $A_\alpha$ such that
$|D^\alpha\psi_i(x)|\le A_\alpha ({\rm diam}\, Q_i)^{-|\alpha|}$ holds for all $i\in\N$ and all $x\in\R^n.$
\end{enumerate}
\end{lemma}

If $a$ is a Lipschitz function on the Lipschitz boundary $\Gamma$ of $\Omega$, then
the Whitney extension operator ${\rm Ext}$ is defined by
\begin{equation}\label{eq:ext:def}
{\rm Ext}\, a(x)=\begin{cases}a(x),&x\in\Gamma,\\
\sum_{i}\mu_{i}\psi_i(x),\quad &x\in \Omega,
\end{cases}
\end{equation}
where we use the notation of Lemma \ref{lem:whitney}
and $\mu_i:=\frac{1}{\sigma(\gamma Q_i\cap \Gamma)}\int_{\gamma Q_i\cap \Gamma}a(y)d\sigma(y)$ with the number 
$\gamma>0$ as described in Lemma \ref{lem:whitney}.
It satisfies $\tr\circ {\rm Ext}\, a=a$ for $a$ Lipschitz continuous on $\Gamma$. This follows directly from the celebrated
Whitney's extension theorem (cf. \cite[p. 23]{JW84}) as $\Gamma$ is a closed set if $\Omega$ is a bounded Lipschitz domain.

\begin{lemma}\label{lemext}
Let $a$ be a Lipschitz function on the Lipschitz boundary $\Gamma$ of $\Omega$.
Then ${\rm Ext}\, a\in C^\infty(\Omega)$ and
\begin{equation}\label{todo1}
\max_{|\alpha|=k} |D^\alpha {\rm Ext}\, a(x)|\le c_k \delta(x)^{1-k}\cdot\|a|{\rm Lip}(\Gamma)\|,\quad k\in\N, \quad x\in \Omega.
\end{equation}
Here, $\delta(x)$ is the distance of $x$ to $\Gamma$ and $c_k$ depends only on $k$ and $\Omega$.
\end{lemma}
\begin{proof}
First, let us note that
$$
D^{\alpha} {\rm Ext}\, a(x)=\sum_{i}\mu_i D^\alpha\psi_i(x), \qquad x\in \Omega, \quad \alpha \in \N_0^n,\quad |\alpha|=k.
$$
By Lemma \ref{lem:whitney}  we have for every $x\in\Omega$
$$
|D^\alpha \psi_i(x)| \le c_k \delta(x)^{-k},\qquad |\alpha|=k,
$$
and
$$
\sum_i D^\alpha\psi_i(x)=D^\alpha \sum_i \psi_i(x)=0.
$$
Furthermore, the Lipschitz continuity of $a$ implies
\begin{equation}\label{eq:finalJV1}
|\mu_{i}-\mu_{j}|\lesssim \delta(x)\cdot \|a|{\rm Lip}(\Gamma)\|
\end{equation}
for $x\in \supp \psi_{i}\cap \supp\psi_{j}.$  
To justify \eqref{eq:finalJV1}, we consider natural numbers $i$ and $j$ with $x\in \supp \psi_{i}\cap \supp\psi_{j}$,
chose any $x_i\in \gamma Q_i\cap \Gamma$ and $x_j\in \gamma Q_j\cap \Gamma$ and calculate
\begin{align*}
|\mu_i-\mu_j|& \le
\left|\frac{1}{\sigma(\gamma Q_i\cap \Gamma)}\int_{\gamma Q_i\cap \Gamma}a(x)d\sigma(x)-a(x_i)\right|+
|a(x_i)-a(x_j)|+
\left|a(x_j)-\frac{1}{\sigma(\gamma Q_j\cap \Gamma)}\int_{\gamma Q_j\cap \Gamma}a(x)d\sigma(x)\right|\\
&\le \|a|{\rm Lip}(\Gamma)\|\cdot\left\{{\rm diam}(\gamma Q_i\cap \Gamma)+|x_i-x_j|+{\rm diam}(\gamma Q_j\cap \Gamma)\right\}\\
&\lesssim \|a|{\rm Lip}(\Gamma)\|\cdot\left\{{\rm diam}(Q_i)+|x_i-x|+|x-x_j|+{\rm diam}(Q_j)\right\}
\lesssim \delta(x)\cdot \|a|{\rm Lip}(\Gamma)\|.
\end{align*}

Let us now fix $x\in\Omega$ and let us denote
by $\{i_1,\dots,i_N\}$, $N\leq N_0$, the indices for which $x$ lies in the support of $\psi_{i}.$ Then we write
\begin{align*}
\left|\sum_{j=1}^N\mu_{{i_j}} D^\alpha\psi_{i_j}(x)\right|&\le
\left|\sum_{j=1}^N(\mu_{i_j}-\mu_{i_1}) D^\alpha\psi_{i_j}(x)\right|+\left|\sum_{j=1}^N\mu_{i_1} D^\alpha\psi_{i_j}(x)\right|\\
&\le \sum_{j=1}^N|\mu_{i_j}-\mu_{i_1}|\cdot|D^\alpha\psi_{i_j}(x)|\lesssim \delta(x)^{1-k}\cdot \|a|{\rm Lip}(\Gamma)\|.
\end{align*}
\end{proof}

\remark{
Let $a$ be a function defined on $\Gamma$ as in Lemma \ref{lemext} with ${\rm diam}\,(\supp a)\le 1$.
Then the extension operator from Lemma \ref{lemext} may be combined with a multiplication with a smooth cut-off function.
This ensures, that \eqref{todo1} still holds and, in addition, ${\rm diam}\,(\supp\,{\rm Ext}\, a)\lesssim 1$. 
}

The following lemma describes a certain geometrical property of Lipschitz domains, which shall be useful later on.
It resembles very much the notion of Minkowski content, cf. \cite{Fa90}.

\begin{lemma}\label{lem:geom}
Let $\Omega$ be a bounded Lipschitz domain and let $k\in\N$. Let $h\in\R^n$ with $0<|h|\le 1$ and put 
$\Omega^h=\{x\in\Omega: [x,x+kh]\subset\Omega\}$. Furthermore, for $j\in\N_0$ we define 
$\Omega^h_j=\{x\in\Omega^h: 2^{-j}\le \min_{y\in[x,x+kh]}\delta(y)\le 2^{-j+1}\}$, where $\delta(y)={\rm dist}(y,\Gamma)$.  Then
\begin{equation}\label{eq:geom1}
|\Omega^h_j|\lesssim 2^{-j}
\end{equation}
with a constant independent of $j$ and $h$.
\end{lemma}

\begin{proof}
To simplify the notation, we shall assume that $\Omega$ is a simple Lipschitz domain of the type
$\Omega=\{(x',x_n)=(x_1,\dots,x_{n-1},x_n)\in\R^n: x_n>\psi(x'), |x'|<1\}$, where $\psi$ is a Lipschitz function, 
and we identify $\Gamma$ with $\{(x',x_n):x_n=\psi(x'), |x'|<1\}$. \\


\underline{Step 1:} First, let us observe that 
\beq\label{eq:geomprf1}
{\rm dist}\,(x,\Gamma)\approx (x_n-\psi(x')) \quad\text{for}\quad x=(x',x_n)\in\Omega
\eeq
and the constants in this equivalence depend only on the Lipschitz constant of $\psi$.
The simple proof of this fact is based on the inner cone property of Lipschitz domains. We refer
to \cite[Chapter VI, Section 3.2, Lemma 2]{stein} for details.\\

\begin{minipage}{.35\textwidth}
\underline{Step 2:} \\

Let $j\in\N_0$ and $0<|h|\le 1$ be fixed and let $$y=(y',y_n)\in \Omega^h_j$$
and let also $$\tilde y=(y',\tilde y_n)\in \Omega^h_j$$ with $\tilde y_n>y_n$. 

As $\tilde y\in\Omega^h_j$, there is a $t_0\in [0,k]$ such that ${\rm dist}(\tilde y+t_0h,\Gamma)\le 2^{-j+1}$. 

\end{minipage}\hfill
\begin{minipage}{0.65\textwidth}
\begin{psfrags}
 	\psfrag{R}{$\mathbb{R}^{n-1}$}
	\psfrag{r}{$\real$}
	\psfrag{h}{$h$}
	\psfrag{O}{$\Omega$}
	\psfrag{G}{$\Gamma$}
	\psfrag{a}{\small{$\tilde{y}=(y',\tilde{y}_n)$}}
	\psfrag{b}{\small{${y}=(y',{y}_n)$}}
	\psfrag{c}{\small{$\tilde{y}+kh$}}
	\psfrag{d}{\small{$y+kh$}}
	\psfrag{e}{\small{$y+t_0h$}}
	\psfrag{f}{\small{$\tilde{y}+t_0h$}}
	\psfrag{y}{$y'$}
	\psfrag{j}{$2^{-j}$}
 	{\includegraphics[width=10cm]{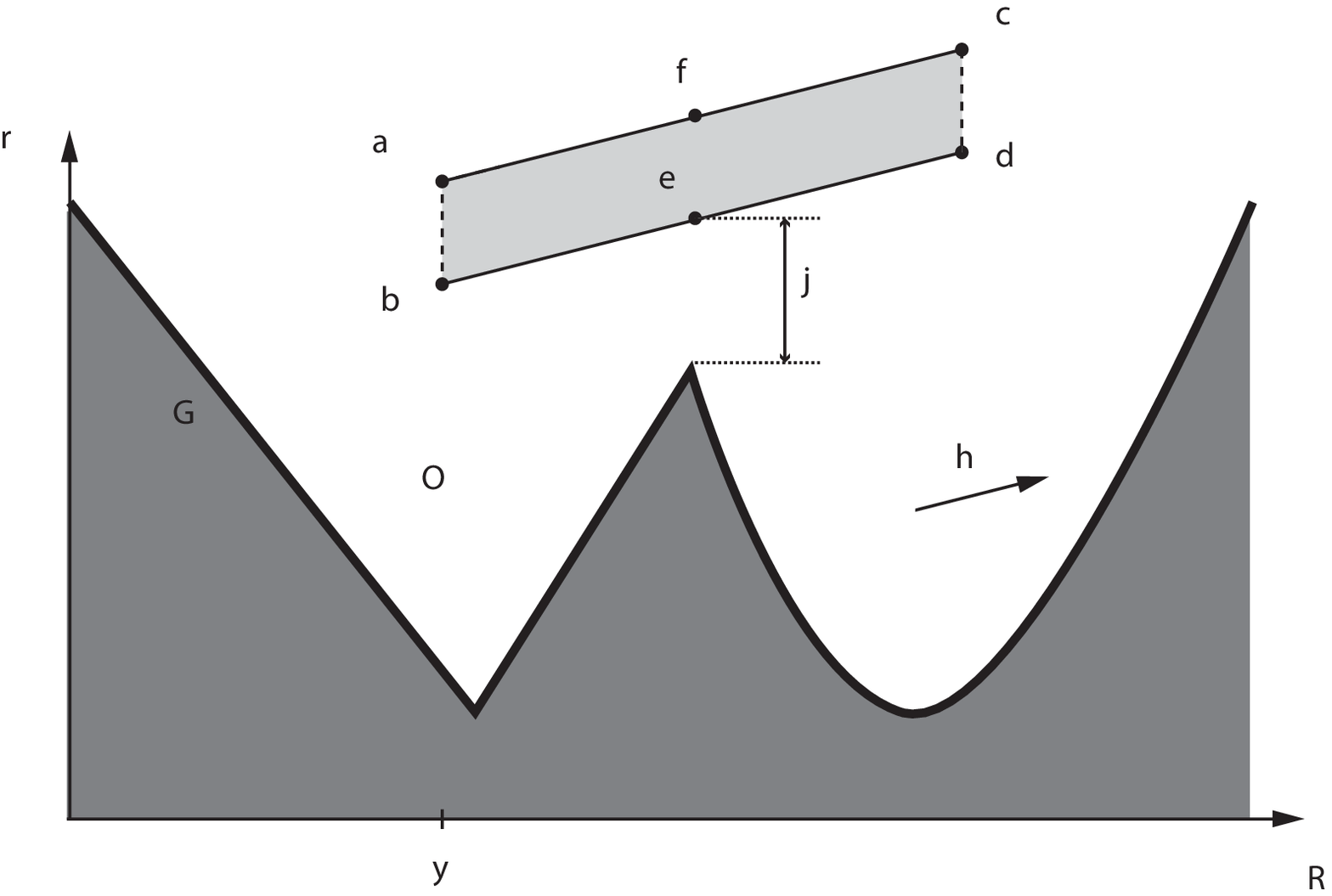}}
\end{psfrags}
\end{minipage}\\

Then we use $\psi(y'+t_0h)<t_0h_n+y_n$ (which follows from $y\in\Omega^h$ and $y+t_0h\in\Omega$) and \eqref{eq:geomprf1} to get
\begin{align}
\notag \tilde y_n-y_n&=[\tilde y_n+t_0h_n-\psi(y'+t_0h')]+[\psi(y'+t_0h')-t_0h_n-y_n]\\
\label{eq:geomprf2}&\lesssim {\rm dist}(\tilde y+t_0h,\Gamma)\lesssim 2^{-j}.
\end{align}

\underline{Step 3:} Using \eqref{eq:geomprf2}, we observe that the set 
$\Omega(x')=\{x_n\in\R:(x',x_n)\in\Omega^h_j\}$ has for every $|x'|<1$ length
smaller then $c\,2^{-j}.$ From this, the inequality \eqref{eq:geom1} quickly follows.
\end{proof}

We shall use this geometrical observation together with the extension operator \eqref{eq:ext:def} to prove the following.

\begin{lemma}\label{lem:central}
Let $\Omega$ be a bounded Lipschitz domain and let $\Gamma$ be its boundary. Let $a$ be a Lipschitz function on $\Gamma$.
Let $0<p\le\infty$, $0<s<\infty$ and $k\in\N$ with $0<s< k<1/p+1$. Then the extension operator defined by \eqref{eq:ext:def} satisfies
\begin{equation}\label{central:2}
\|{\rm Ext}\, a|\mathbf{B}^{s}_{p,p}(\Omega)\|\lesssim \|a|{\rm Lip}(\Gamma)\|
\end{equation}
with the constant independent of $a\in {\rm Lip}(\Gamma)$.
\end{lemma}
\begin{proof}
Using the characterization by differences, we obtain
\begin{align*}
\|{\rm Ext}\, a|\mathbf{B}^{s}_{p,p}(\Omega)\|&\lesssim \|{\rm Ext}\, a|\mathbf{B}^{s'}_{p,\infty}(\Omega)\|\\
&\lesssim \|{\rm Ext}\,a|L_p(\Omega)\|+\sup_{0<|h|\le 1}|h|^{-s'}\|\Delta_h^k {\rm Ext}\, a(\cdot,\Omega)|L_p(\Omega)\|,
\end{align*}
for $s'>0$ with $s<s'<k$. 
Furthermore, we observe that one may modify the definition of $\Delta_h^r f(x,\Omega)$ given in \eqref{eq:finalJV2}
to be zero also if the whole segment $[x,x+kh]$ is not a subset of $\Omega$. This follows by a detailed inspection
of \cite[Section 2.5.12]{T-F1} as well as \cite{Di03} and \cite{dVS93}, which are all based on the integration in cones.

Using the definition of $\mu_i$, the first term may be estimated easily as
$$
\|{\rm Ext}\,a|L_p(\Omega)\| \lesssim \|{\rm Ext}\,a|L_\infty(\Omega)\|\le \|a|L_\infty(\Gamma)\|.
$$

To estimate the second term, we shall need the following relationship between differences and derivatives.
If $f\in C^k(\R^n)$ and $x,h\in\R^n$, we put $g(t)=f(x+th)$ for $t\in\R$ and obtain
\begin{equation}\label{eq:HIDR}
\Delta^k_h f(x)=\Delta^k_1 g(0) = \int_0^k g^{(k)}(t) B_k(t)dt,
\end{equation}
where $B_k$ is the standard $B$ spline of order $k$, i.e. the $k$-fold convolution of $\chi_{[0,1]}$ given by \mbox{$B_k =\chi_{[0,1]}*\dots*\chi_{[0,1]}$}.
Although \eqref{eq:HIDR} is a classical result of approximation theory (c.f. \cite[Section 4.7]{DL93}), let us give a short proof 
using Fubini's Theorem and induction over $k$:
\begin{align*}
\Delta^{k+1}_1 g(0)&=\Delta^k_1g(1)-\Delta^k_1g(0)=\int_0^k (g^{(k)}(t+1)-g^{(k)}(t))B_k(t)dt\\
&=\int_0^k B_k(t)\int_t^{t+1}g^{(k+1)}(u)du\, dt=\int_0^{k+1} g^{(k+1)}(u)\int_{u-1}^{u} B_k(t)dtdu=\int_0^{k+1}g^{(k+1)}(u)B_{k+1}(u)du.
\end{align*}
Hence if $[x,x+kh]\subset \Omega$ for some $x\in \Omega$, we obtain
$$
|\Delta^k_h {\rm Ext}\, a(x,\Omega)| \lesssim |h|^k\int_0^k \max_{|\alpha|=k} |D^\alpha {\rm Ext}\, a (x+th)|\cdot B_k(t) dt
\lesssim |h|^k\cdot\|a|{\rm Lip}(\Gamma)\|\cdot\int_0^k \delta(x+th)^{1-k}\cdot B_k(t) dt.
$$
Let us fix $h\in\R^n$ with $0<|h|\le 1$ and let us denote $\Omega^h=\{x\in\Omega: [x,x+kh]\subset\Omega\}$ as in Lemma \ref{lem:geom}. We obtain
\begin{align*}
|h|^{-s'}\|\Delta_h^k {\rm Ext}\, a(\cdot,\Omega)|L_p(\Omega)\|&\lesssim
|h|^{k-s'}\|a|{\rm Lip}(\Gamma)\|\left(\int_{\Omega^h}\left(\int_0^k \delta(x+th)^{1-k}\cdot B_k(t) dt\right)^p dx\right)^{1/p}\\
&\lesssim \|a|{\rm Lip}(\Gamma)\|\left(\int_{\Omega^h} \max_{y\in[x,x+kh]}\delta(y)^{(1-k)p} dx\right)^{1/p}\\
&\lesssim \|a|{\rm Lip}(\Gamma)\|\left(\sum_{j=0}^\infty 2^{-j(1-k)p} |\Omega^h_j|\right)^{1/p}.\\
\end{align*}
This, together with Lemma \ref{lem:geom} and with $k<1/p+1$ finishes the proof.

\end{proof}
 
\begin{lemma}\label{ext-atom}
Let $0<s'<1$ be fixed. 
There is a non-linear extension operator (denoted by $\mathbf{Ext}$), 
which extends $\Lip^\Gamma$-atoms $a_{j,m}$ to $(s'+1/p,p)$-atoms on $\R^n$. 
\end{lemma}

\proofstart
As the definition of $\Lip^\Gamma$-atoms as well as the definition of $(s'+1/p,p)$-atoms works with
$a_j(2^{-j}\cdot)$, by homogeneity arguments it is enough to prove 
\begin{equation}\label{central}
\|\mathbf{Ext}\, a_{0,m}|\mathbf{B}^{s'+1/p}_{p,p}(\R^n)\|\lesssim \|a_{0,m}|{\rm Lip}(\Gamma)\|
\end{equation}
for $\Lip^\Gamma$-atoms $a_{j,m}$ with $j=0$. 
First we show  that
\begin{equation}\label{central'}
\|{\rm Ext}\, a_{0,m}|\mathbf{B}^{s'+1/p}_{p,p}(\Omega)\|\lesssim \|a_{0,m}|{\rm Lip}(\Gamma)\|
\end{equation}
for the extension operator constructed in \eqref{eq:ext:def}. Let $0<s'<1$ and $0<p\le \infty$. We observe, that Lemma \ref{lem:central} implies
\eqref{central'} for all $0<s'<1$ for which there is a $k\in\N_0$ with \\
\begin{minipage}{0.46\textwidth}
$$s'+1/p<k<1+1/p.$$ In the diagram aside these points correspond to all $(s',\frac 1p)$ in the gray-shaded triangles. \\
Then Lemma \ref{lem1} yields \eqref{central'} for all $0<s'<1$ and $0<p\le\infty$
with $s_0=s_1=s'$ and $p_0<p<p_1$ chosen in an appropriate way, see the attached diagram.
\end{minipage}\hfill
\begin{minipage}{0.50\textwidth}
\begin{psfrags}
 	\psfrag{0}{$0$}
	\psfrag{1}{$1$}
	\psfrag{z}{$1$}
	\psfrag{a}{\small{$\left(s, \frac{1}{p_0}\right)$}}
	\psfrag{b}{\small{$\left(s, \frac{1}{p_1}\right)$}}
	\psfrag{c}{\small{$\left(s, \frac{1}{p}\right)$}}
	\psfrag{d}{\small{$s+\frac 1p=1$}}
	\psfrag{e}{\small{$s+\frac 1p=2$}}
	\psfrag{f}{\small{$s+\frac 1p=3$}}
	\psfrag{s}{$s$}
	\psfrag{p}{$\frac 1p$}
 	{\includegraphics[width=8.4cm]{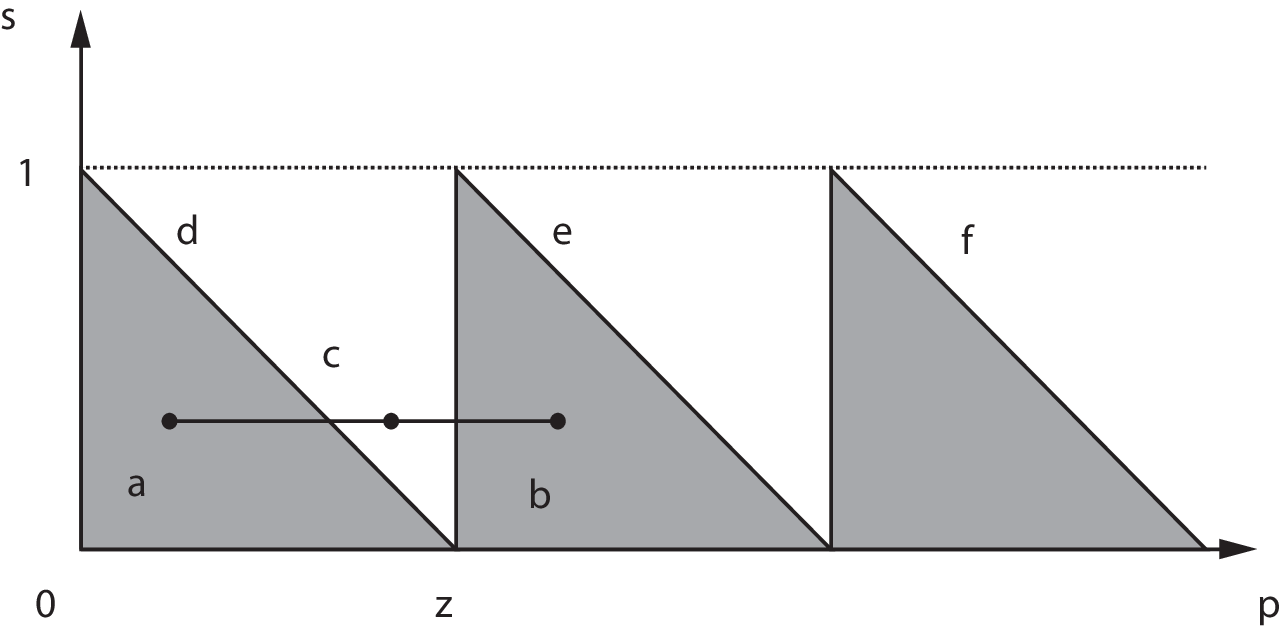}}
\end{psfrags}
\end{minipage}\\

Finally, by Remark \ref{extr_char}, we know that there is a function (denoted by $\mathbf{Ext}\,a_{0,m}$),
such that
$$
\|\mathbf{Ext}\,a_{0,m}|\mathbf{B}^{s'+1/p}_{p,p}(\R^n)\|\lesssim\|{\rm Ext}\,a_{0,m}|\mathbf{B}^{s'+1/p}_{p,p}(\Omega)\|.
$$
This together with \eqref{central'} finishes the proof of \eqref{central}.
\proofend
 
We are now able to complete the proof of the missing part of the trace theorem.
 
\begin{theorem}\label{B-trace-2}
Let $n\geq 2$ and $\Omega$ be a bounded Lipschitz domain with boundary $\Gamma$. Then for $0<s<1$ and $0<p,q\leq\infty$
there is a bounded non-linear extension operator
\beq\label{B-trace-th2}
{Ext}:\mathbf{B}^{s}_{p,q}(\Gamma)\longrightarrow \mathbf{B}^{s+\frac 1p}_{p,q}(\Omega).
\eeq
\end{theorem}

\proofstart
Let $f\in{\mathbf B}^s_{p,q}(\Gamma)$ with optimal decomposition in the sense of Theorem \ref{th:at-bd-1}
\begin{equation}\label{eq:decomp1}
f(x)=\sum_{j=0}^\infty \sum_{m\in\Z^{n}}\lambda_{j,m}a^{\Gamma}_{j,m}(x),
\end{equation}
where $a^\Gamma_{j,m}$ are $\Lip^\Gamma$-atoms, \eqref{eq:decomp1} converges in $L_p(\Gamma)$,
and $\|f|{\mathbf B}^s_{p,q}(\Gamma)\|\sim \|\lambda|b^s_{p,q}(\Gamma)\|.$

We use the extension operator constructed in Lemma \ref{ext-atom} and define by
\beq\label{eq:prf1}
{Ext}\,f:=\sum_{j=0}^\infty \sum_{m\in\Z^{n}}\lambda_{j,m}(\mathbf{Ext}\, a^\Gamma_{j,m})|_\Omega
\eeq
an atomic decomposition of $f$ in the space ${\mathbf B}^{s+1/p}_{p,q}(\Omega)$ with non-smooth $(s'+1/p,p)$-atoms $\mathbf{Ext}\, a^\Gamma_{j,m}$,
where $s<s'<1.$
The convergence of \eqref{eq:prf1} in $L_p(\Omega)$ follows in the same way as in the proof of Step 3 of Theorem \ref{th:at-bd-1}.

Together with $\|\lambda|b^s_{p,q}(\Gamma)\|\sim\|\lambda|b^{s+1/p}_{p,q}(\Omega)\|$, this shows that
$$\|{Ext}\, f|{\mathbf B}^{s+1/p}_{p,q}(\Omega)\|\lesssim \|\lambda|b^{s+1/p}_{p,q}(\Omega)\|\sim \|\lambda|b^s_{p,q}(\Gamma)\|<\infty
$$
is bounded.\\
\smallskip

Theorems \ref{B-trace-1} and \ref{B-trace-2} together now allow us to state the general result for traces on Lipschitz domains without any
restrictions on the parameters $s,p$ and $q$.

\begin{theorem}\label{B-trace-lip}
Let $n\geq 2$ and $\Omega$ be a bounded Lipschitz domain with boundary $\Gamma$. Then for $0<s<1$ and $0<p,q\leq\infty$,
\beq\label{B-trace-th1}
\tr \mathbf{B}^{s+\frac 1p}_{p,q}(\Omega)=\mathbf{B}^{s}_{p,q}(\Gamma).
\eeq
\end{theorem}

 
The above Theorem extends the trace results obtained in \cite[Th.~2.4]{sch10} from $C^k$ domains with $k>s$ to Lipschitz domains. \\
Furthermore, the trace results for  spaces of Triebel-Lizorkin type carry over as well to the case of Lipschitz domains. The proof follows  
\cite[Th.~2.6]{sch10} where the independence of the trace on $q$ was established for F-spaces. 
Let us mention that the sequence spaces $f^s_{p,q}(\Omega)$ are defined similarly as $b^s_{p,q}(\Omega)$,
cf. Definition \ref{def-seq-dom}, with $\ell_p$ and $\ell_q$ summation interchanged. The corresponding function spaces 
(denoted by $\mathfrak{F}^{s}_{p,q}(\Omega)$) are then defined as in Definition \ref{intr_char}.

The main ingredient in the study of traces for Triebel-Lizorkin spaces ${\mathfrak F}^s_{p,q}(\Omega)$ is then 
the fact that the corresponding sequence spaces
$f^s_{p,q}(\Gamma)$ are independent of $q$,
\beq\label{indep-on-q}
f^s_{p,q}(\Gamma)=b^s_{p,p}(\Gamma).
\eeq
A proof may be found in  \cite[Prop.~9.22, p.~394]{T-F3} for $\Gamma$ being a compact porous set in $\rn$ with \cite{frj90} as an important forerunner. 
In \cite[Prop.~3.6]{T-wave-dom}
it is shown that the boundaries $\partial\Omega=\Gamma$ of $(\varepsilon,\delta)$-domains $\Omega$  are  porous. Therefore, this result is also 
true for boundaries of Lipschitz domains.
 
For completeness we  state the trace results for F-spaces below.
 
\begin{corollary}\label{F-trace-lip}
Let $0<p<\infty$, $0<q\leq\infty$, $0<s<1$, and let $\Omega\subset \rn$ be a bounded Lipschitz domain with boundary $\Gamma$.
Then
\beq\label{trace-F-dom}
\tr \mathfrak{F}^{s+\frac 1p}_{p,q}(\Omega)=\mathbf{B}^{s}_{p,p}(\Gamma).
\eeq
\end{corollary}


\subsection{The limiting case}

We briefly discuss what happens in the limiting case $s=0$. In \cite[Th.~2.7]{sch11a} traces for Besov and Triebel-Lizorkin spaces on
$d$-sets $\Gamma$, $0<d<n$, were studied. In particular, it was shown that for $0<p<\infty$ and $0<q\leq\infty$,
\beq
\tr \mathbf{B}^{\frac{n-d}{p}}_{p,q}(\rn)=L_p(\Gamma), \qquad 0<q\leq\min(1,p),
\eeq
 and
\beq
\tr \mathfrak{F}^{\frac{n-d}{p}}_{p,q}(\rn)=L_p(\Gamma), \qquad 0<p\leq 1.
\eeq
Since the boundary $\Gamma$ of a Lipschitz domain $\Omega$ is a $d$-set with $d=n-1$ the results  follow almost immediately from 
these previous results, using the fact that the B- and F-spaces on domains $\Omega$ are defined as restrictions of the corresponding spaces on
$\rn$, cf. Remark \ref{extr_char}.

\begin{corollary}\label{B-trace-lim-n}
Let $\Omega$ be a bounded Lipschitz domain with boundary $\Gamma$. Furthermore, let $0<p<\infty$ and $0<q\leq\infty$.
\begin{itemize}
\item[(i)] Then
\beq\label{B-tr-lim-case}
\tr \mathbf{B}^{\frac{1}{p}}_{p,q}(\Omega)=L_p(\Gamma), \qquad 0<q\leq\min(1,p).
\eeq
\item[(ii)] Furthermore,
\beq\label{F-tr-lim-case}
\tr \mathfrak{F}^{\frac{1}{p}}_{p,q}(\Omega)=L_p(\Gamma), \qquad 0<p\leq 1.
\eeq
\end{itemize}
\end{corollary}


\section{Pointwise multipliers in function spaces}

As an application we now use our results on non-smooth atomic decompositions to deal with pointwise multipliers in the respective function spaces. \\

{A function} $m$ in $L_{\min(1,p)}^{loc}(\rn)$ is called a {\em pointwise multiplier} for $\Bd(\rn)$ if\\
\[f\mapsto mf\]
generates a bounded map in $\Bd(\rn)$. The collection of all multipliers for $\Bd(\rn)$ is denoted by $M(\Bd(\rn))$. In the following, let $\psi$ stand for a non-negative $C^{\infty}$ function with 
\beq\label{psi-1}
\supp \psi\subset \{y\in\rn: |y|\leq \sqrt{n}\}
\eeq
and 
\beq\label{psi-2}
\sum_{l\in\zn}\psi(x-l)=1, \qquad x\in\rn.
\eeq

\begin{definition}\label{def-Bselfs}
Let $s>0$ and $0<p,q\leq \infty$. We define the space $\Bselfs(\rn)$ to be the set of all $f\in L_{\min(1,p)}^{loc}(\rn)$ such that 
\beq\label{Bselfs}
\|f|\mathbf{B}^s_{p,q,\selfs}(\rn)\|:=\sup_{j\in\nat_0, l\in\zn}\|\psi(\cdot -l)f(2^{-j}\cdot)|\Bd(\rn)\|
\eeq
is finite. 
\end{definition}

\remark{

The study of pointwise multipliers is one of the key problems of the theory of function spaces. As far as classical Besov spaces and (fractional) Sobolev spaces  with $p>1$ are concerned we refer to \cite{mazya}, \cite{MaSh85}, and \cite{MaSh09}. Pointwise multipliers in general spaces $\B(\rn)$ and $\F(\rn)$ have been studied in great detail in \cite[Ch.~4]{RS96}.\\ 
Selfsimilar spaces were first introduced in \cite{tri03} and then considered in \cite[Sect.~2.3]{T-F3}.  Corresponding results for anisotropic function spaces may be found in \cite{MPP07}. We also mention their forerunners, the uniform spaces $\mathbf{B}^s_{p,q,\mathrm{unif}}(\rn)$, studied in detail in \cite[Sect.~4.9]{RS96}. As stated in \cite{KS02}, for these spaces it is known that 
\[
M(\Bd(\rn))=\mathbf{B}^s_{p,q,\mathrm{unif}}(\rn), \qquad 1\leq p\leq q \leq \infty, \quad s>\frac np,
\]
cf. \cite{SS99} concerning the proof.  
Selfsimilar spaces are also closely connected with pointwise multipliers. We shall use the abbreviation $$\mathbf{B}^s_{p,\selfs}(\rn):=\mathbf{B}^s_{p,p,\selfs}(\rn).$$ 
One can easily show 
\beq\label{emb_Bselfs}
\Bselfs(\rn)\hookrightarrow L_{\infty}(\rn).
\eeq
To see this applying homogeneity gives
\[
\|\psi(\cdot -l)f(2^{-j}\cdot)|\Bd(\rn)\|\sim 2^{j\frac np}\|\psi(2^j\cdot -l)f|L_p(\rn)\|+2^{-j(s-\frac np)}\left(\int_0^1 t^{-sq}{\omega_r(\psi(2^j\cdot -l)f,t)_p}^q\frac{\ud t}{t}\right)^{1/q}
\]
uniformly for all $j\in\nat_0$ and $l\in\zn$. Consequently,
\beq\label{leb-p}
2^{jn}\int_{\rn}|\psi(2^jy-l)|^p|f(y)|^p\ud y\leq c\|f|\Bselfs(\rn)\|^p.
\eeq
Thus, the right-hand side of \eqref{leb-p} is just a uniform bound for $|f(\cdot)|^p$ at its Lebesgue points, cf. \cite[Cor. p.13]{stein93}, which proves the desired embedding  \eqref{emb_Bselfs}.  

}

\begin{definition}
Let $s>0$ and $0<p,q\leq \infty$. We define
\[
\mathbf{B}^{s+}_{p,q,\selfs}(\rn):=\bigcup_{\sigma>s}\mathbf{B}_{p,q,\selfs}^{\sigma}(\rn).
\]

\end{definition}

We have the following relation between pointwise multipliers and self-similar spaces.

\begin{theorem}\label{th:multipliers-1}
Let $s>0$ and $0<p,q\leq \infty$. Then
\bit
\item[(i)] $\ds \mathbf{B}^{s+}_{p,q,\selfs}(\rn)\subset M(\Bd(\rn))\hookrightarrow \mathbf{B}^{s}_{p,q,\selfs}(\rn)$
\item[(ii)] Additionally, if $0<p\leq 1$,
\[
M(\mathbf{B}^s_p(\rn))=\mathbf{B}^{s}_{p,\selfs}(\rn). 
\] 
\eit

\end{theorem}

\proofstart
We first prove the right-hand side embedding in (i). Let $m\in M(\Bd(\rn))$. An application of the homogeneity property from Theorem \ref{hom-B} yields 
\begin{eqnarray*}
\|\psi(\cdot -l)m(2^{-j}\cdot)|\Bd(\rn)\|
&\sim & 2^{-j(s-\frac np)}\|\psi(2^j\cdot -l)m|\Bd(\rn)\|\\
&\lesssim & 2^{-j(s-\frac np)} \|m|M(\Bd(\rn))\|\cdot\|\psi(2^j\cdot -l)|\Bd(\rn)\|\\
& = & 2^{-j(s-\frac np)}\|m|M(\Bd(\rn))\|\cdot\|\psi(2^j\cdot )|\Bd(\rn)\|\\
&\sim& \|m|M(\Bd(\rn))\|\|\psi|\Bd(\rn)\|\lesssim \|m|M(\Bd(\rn))\|
\end{eqnarray*}
for all $l\in\zn$, $j\in\nat_0$, and hence,
\begin{eqnarray*}
\|m|\Bselfs(\rn)\|&=&\sup_{j\in\nat_0, l\in\zn}\|\psi(\cdot -l)m(2^{-j})|\Bd(\rn)\|\\
&\lesssim & \|m|M(\Bd(\rn))\|. 
\end{eqnarray*}
We make use of the non-smooth atomic decompositions for $\Bd(\rn)$ from Theorem \ref{th:atom-dec} in order to prove the first inclusion in (i). Let $m\in \mathbf{B}^{\sigma}_{p,q,\selfs}$ with $\sigma>s$.  Let $f\in \Bd(\rn)$ with optimal smooth atomic decomposition
\beq\label{m-00}
f=\sum_{j=0}^{\infty}\sum_{l\in\zn}\lambda_{j,l}a_{j,l}\quad \text{ with }\quad\|f|\Bd(\rn)\|\sim \|\lambda|b^s_{p,q}\|,
\eeq
where $a_{j,m}$ are $K$-atoms with $K>\sigma$. Then 
\beq\label{m-0}
mf=\sum_{j=0}^{\infty}\sum_{l\in\zn}\lambda_{j,l}\left( m a_{j,l}\right),
\eeq
and we wish to prove that, up to normalizing constants, the $\ ma_{j,l}\ $ are $(\sigma,p)$-atoms. The support condition is obvious:
\[
\supp m a_{j,l}\subset \supp a_{j,l}\subset dQ_{j,l}, \qquad j\in\nat_0, l\in\zn.
\]
If $l=0$ we put $a_j=a_{j,l}$. Note that 
\[
\supp a_{j}(2^{-j})\subset \{y: |y_i|\leq \frac d2\}
\]
and we can assume that 
\[
\psi(y)>0 \qquad \text{if}\quad y\in\{x:|x_i|\leq d\}.
\]
Then -- using  multiplier assertions from \cite[Prop.~2.15(ii)]{sch10} --  we have for any $g\in \mathbf{B}^{\sigma}_{p,q}(\rn)$, 
\begin{eqnarray*}
\|a_j(2^{-j})\psi^{-1}g|\mathbf{B}^{\sigma}_{p,q}(\rn)\|
&\lesssim & \|a_j(2^{-j})\psi^{-1}|C^K(\rn)\|\|g|\mathbf{B}^{\sigma}_{p,q}(\rn)\|\\
&\lesssim &  \|g|\mathbf{B}^{\sigma}_{p,q}(\rn)\|
\end{eqnarray*}
and hence
\beq\label{m-1}
\|a_j(2^{-j})\psi^{-1}|M(\mathbf{B}^{\sigma}_{p,q}(\rn))\|\lesssim 1, \qquad j\in \nat_0.
\eeq
By \eqref{m-1} and the homogeneity property we then get, for any $\sigma>\sigma'>s$ and $j\in\nat_0$,
\begin{eqnarray}
\|(m a_j)(2^{-j}\cdot)|\mathbf{B}^{\sigma'}_{p}(\rn)\|&\lesssim &\|m(2^{-j}\cdot)a_j(2^{-j}\cdot)|\mathbf{B}^{\sigma}_{p,q}(\rn)\|\notag\\
&\lesssim & \|a_j(2^{-j}\cdot)\psi^{-1}|M(\mathbf{B}^{\sigma}_{p,q}(\rn))\| \|m(2^{-j}\cdot)\psi|\mathbf{B}^{\sigma}_{p,q}(\rn)\|\notag\\
&\lesssim & \|m(2^{-j}\cdot)\psi|\mathbf{B}^{\sigma}_{p,q}(\rn)\|\label{m-2}.
\end{eqnarray} 
In the case of $a_{j,l}$ with $l\in\zn$ one arrives at \eqref{m-2} with $a_{j,l}$ and $\psi(\cdot -l)$ in place of $a_j$ and $\psi$, respectively. Hence
\begin{eqnarray}
\|m a_{j,l}(2^{-j}\cdot)|\mathbf{B}^{\sigma'}_{p}(\rn)\|& \lesssim & \sup_{j,l}\|m(2^{-j}\cdot)\psi(\cdot - l)|\mathbf{B}^{\sigma}_{p,q}(\rn)\|\notag\\
& = & \|m|\mathbf{B}^{\sigma}_{p,q,\selfs}(\rn)\|, \qquad j\in\nat_0, l\in\zn,
\end{eqnarray}
and therefore, $m a_{j,l}$ is a $(\sigma',p)$-atom where $\sigma'>s$. By Theorem \ref{th:atom-dec}, in view of \eqref{m-0}, $mf\in\mathbf{B}^{s}_{p,q}(\rn)$ and 
\begin{eqnarray*}
\|mf|\mathbf{B}^{s}_{p,q}(\rn)\|\leq \|\lambda|b^s_{p,q}\|\|m|\mathbf{B}^{\sigma}_{p,q,\selfs}(\rn)\|\sim \|f|\mathbf{B}^s_{p,q}\|\|m|\mathbf{B}^{\sigma}_{p,q,\selfs}(\rn)\|,
\end{eqnarray*}
which completes the proof of (i). \\

\smallskip 

We now prove (ii). Restricting ourselves to $p=q$, let now $m\in \mathbf{B}^s_{p,\selfs}(\rn)$. 
We can modify \eqref{m-2} by choosing $\sigma'=\sigma=s$,
\begin{eqnarray}
\|(m a_j)(2^{-j}\cdot)|\mathbf{B}^{s}_{p}(\rn)\| &= &\|m(2^{-j}\cdot)a_j(2^{-j}\cdot)|\mathbf{B}^{s}_{p}(\rn)\|\notag\\
&\lesssim & \|a_j(2^{-j}\cdot)\psi^{-1}|M(\mathbf{B}^{s}_{p}(\rn))\| \|m(2^{-j}\cdot)\psi|\mathbf{B}^{s}_{p}(\rn)\|\notag\\
&\lesssim & \|m(2^{-j}\cdot)\psi|\mathbf{B}^{s}_{p}(\rn)\|\label{m-4},
\end{eqnarray} 
yielding for general atoms $a_{j,l}$, 
\begin{eqnarray}
\|m a_{j,l}(2^{-j}\cdot)|\mathbf{B}^{s}_{p,}(\rn)\|& \lesssim & \sup_{j,l}\|m(2^{-j}\cdot)\psi(\cdot - l)|\mathbf{B}^{s}_{p}(\rn)\|\notag\\
& = & \|m|\mathbf{B}^{s}_{p,\selfs}(\rn)\|, \qquad j\in\nat_0, l\in\zn.\label{m-5}
\end{eqnarray}
Since $p\leq 1$, we have that  $\mathbf{B}^s_{p}(\rn)$ is a $p$-Banach space. From \eqref{m-00}, using \eqref{m-0} and \eqref{m-5}, we obtain 
\begin{eqnarray}
\|mf|\mathbf{B}^s_{p}(\rn)\|^p
&\leq & \sum_{j=0}^{\infty}\sum_{l\in\zn}|\lambda_{j,l}|^p2^{j(s-\frac np)p}2^{-j(s-\frac np)p}\|m a_{j,l}|\mathbf{B}^s_{p}(\rn)\|^p\notag\\
&\sim &\|\lambda|b^s_{p,p}\|^p\|(m a_{j,l})(2^{-j}\cdot)|\mathbf{B}^s_{p}(\rn)\|^p\notag\\
&\lesssim &\|\lambda|b^s_{p,p}\|^p\|m|\mathbf{B}^{s}_{p,\selfs}(\rn)\|^p.\label{open-1}
\end{eqnarray}
Hence $m\in M(\mathbf{B}^s_{p}(\rn))$ and, moreover, $\mathbf{B}^s_{p,\selfs}(\rn)\hookrightarrow M(\mathbf{B}^s_p(\rn))$. The other embedding follows from part (i). 
\proofend

\remark{It remains open whether it is possible or not to generalize Theorem \ref{th:multipliers-1}(ii) to the case when $p\neq q$. The problem in the proof given above is the estimate \eqref{open-1}, which only holds if $p=q$. }

\paragraph{Characteristic functions as multipliers}

The final part of this work is devoted to the question in which function spaces the characteristic function $\chi_{\Omega}$ of a domain $\Omega\subset\rn$ is a pointwise multiplier. We contribute to this question mainly as an application of Theorem \ref{th:multipliers-1}. The results shed some light on a relationship between some fundamental notion of fractal geometry and pointwise multipliers in function spaces. For complementary remarks and studies in this direction we refer to \cite{tri03}. \\
There are further considerations of a similar kind in the literature, asking for geometric conditions on the domain $\Omega$ such that the corresponding characteristic function $\chi_{\Omega}$ provides multiplier properties, cf. \cite{Gu1, Gu2}, \cite{frj90}, and \cite[Sect.~4.6.3]{RS96}.

\begin{definition}\label{def-h-set}
Let $\Gamma$ be a non-empty compact set in $\rn$. Let $h$ be a positive non-decreasing function on the interval $(0,1]$. Then $\Gamma$ is called a $h$-set, if there is a finite Radon measure $\mu\in\rn$ with 
\beq\label{h-set}
\supp \mu=\Gamma\quad \text{ and } \quad \mu(B(\gamma,r))\sim h(r), \qquad \gamma\in \Gamma, \; 0<r\leq 1. 
\eeq
\end{definition}

\remark{
A measure $\mu$ with \eqref{h-set} satisfies the so-called  {\em doubling condition}, meaning there is a constant $c>0$ such that 
\beq\label{dc}
\mu(B(\gamma, 2r))\leq c \mu(B(\gamma,r)), \qquad \gamma\in \Gamma, \; 0<r<1. 
\eeq
We refer to \cite[p. 476]{tri03} for further explanations.
}

\begin{theorem}\label{th:multipliers-2}
Let $\Omega$ be a bounded domain in $\rn$. Moreover, let  $\sigma>0$, $0<p<\infty$, $0<q\leq\infty$, and let $\Gamma=\partial \Omega$ be an $h$-set with 
\beq\label{m-23}
\sup_{j\in\nat_0}\sum_{k=0}^{\infty}2^{k\sigma q}\left(\frac{h(2^{-j})}{h(2^{-j-k})}2^{-kn}\right)^{q/p}<\infty,
\eeq
$($with the usual modifications if $q=\infty$$)$. 
Let $\mathbf{B}^{\sigma}_{p,q,\selfs}(\rn)$ be the spaces defined in \eqref{Bselfs}. Then
\[
\chi_{\Omega}\in \mathbf{B}^{\sigma}_{p,q,\selfs}(\rn).
\]
\end{theorem}

\proofstart
It simplifies the argument, and causes no loss of generality, to assume $\mathrm{diam}\  \Omega<1$. We define
\[
\Omega^k=\{x\in\Omega:2^{-k-2}\leq \mathrm{dist}(x,\Gamma)\leq 2^{-k}\}, \qquad k\in\nat_0.
\]
Moreover, let 
\[
\{\varphi^k_l: k\in\nat_0, \; l=1,\ldots, M_k\}\subset C^{\infty}_0(\Omega)
\]
be a resolution of unity,
\beq \label{m-21}
\sum_{k\in\nat_0}\sum_{l=1}^{M_k}\varphi_l^k(x)=1\qquad \text{if } x\in\Omega,
\eeq
with
\[
\supp \varphi_l^k\subset \{x:|x-x_l^k|\leq 2^{-k}\}\subset \Omega^k
\]
and 
\[
|\uD^{\alpha} \varphi_l^k(x)|\lesssim 2^{|\alpha| k},\qquad  |\alpha|\leq K,
\]
where $K\in\nat$ with $K>\sigma$.  It is well known that resolutions of unity with the required properties exist. We now estimate the number $M_k$ in \eqref{m-21}. Combining the fact that the measure $\mu$ satisfies the doubling condition \eqref{dc} together with \eqref{h-set} we arrive at
\beq\label{m-22}
M_k h(2^{-k})\lesssim 1, \qquad k\in\nat_0. 
\eeq
Since the $\varphi_l^k$ in \eqref{m-21} are $K$-atoms according to Definition \ref{def-atoms}, we obtain
\beq\label{m-22a}
\|\chi_{\Omega}|\mathbf{B}^{\sigma}_{p,q}(\rn)\|^q\leq \sum_{k=0}^{\infty}2^{k(\sigma-n/p)q}M_k^{q/p}\lesssim \sum_{k=0}^{\infty}2^{k\sigma q}\left(\frac{2^{-kn}}{h(2^{-k})}\right)^{q/p}<\infty.
\eeq
This shows that $\chi_{\Omega}\in\mathbf{B}^{\sigma}_{p,q}(\rn)$. We now prove that $\chi_{\Omega}\in \mathbf{B}^{\sigma}_{p,q,\selfs}(\rn)$. We consider the non-negative function $\psi\in C^{\infty}(\rn)$ satisfying \eqref{psi-1} and \eqref{psi-2}. By the definition of self-similar spaces, it suffices to consider 
\[
\chi_{\Omega}(2^{-j}\cdot)\psi,
\]
assuming in addition that $0\in 2^j\Gamma=\{2^j\gamma=(2^{j}\gamma_1,\ldots, 2^j\gamma_n): \gamma\in \Gamma\}$, $j\in\nat$. Let $\mu^j$ be the image measure of $\mu$ with respect to the dilations $y\mapsto 2^j y$. Then we obtain
\[
\mu^j(B(0,\sqrt{n})\cap 2^j\Gamma)\sim h(2^{-j}), \qquad j\in\nat_0.
\]
We apply the same argument as above to $B(0,\sqrt{n})\cap 2^j\Omega$ and $B(0,\sqrt{n})\cap 2^j\Gamma$ in place of $\Omega$ and $\Gamma$, respectively. Let $M^j_k$ be the counterpart of the above number $M_k$. Then
\[
M^j_k h(2^{-j-k})\lesssim h(2^{-j}), \qquad j\in\nat_0, \; k\in\nat_0,
\]
is the generalization of \eqref{m-22} we are looking for, 
which completes the proof. 
\proofend

In view of Theorem \ref{th:multipliers-1} we have the following result.

\begin{corollary}
Let $\Omega$ be a bounded domain in $\rn$. Moreover, let $\sigma>0$, $0<p<\infty$, $0<q\leq \infty$, and let $\Gamma=\partial\Omega$ be a $h$-set satisfying \eqref{m-23}. Then
\[
\chi_{\Omega}\in M(\Bd(\rn))\qquad \text{for }\quad  1<p<\infty, \quad 0<s<\sigma,
\]
and 
\[
\chi_{\Omega}\in M(\mathbf{B}^{\sigma}_{p}(\rn))\qquad \text{for }\quad  0<p\leq 1. 
\]
\end{corollary}

\remark{As for the assertion \eqref{m-23} we mention that 
\[
\sup_{j\in\nat_0, k\in\nat_0}2^{k\sigma}\left(\frac{h(2^{-j})}{h(2^{-j-k})}2^{-kn}\right)^{1/p}<\infty
\]
is the adequate counterpart for $\mathbf{B}^{\sigma}_{p,\infty}(\rn)$. 
In the special case of $d$-sets, which corresponds to $h(t)\sim t^d$, the condition \eqref{m-23} therefore corresponds to 
\[
\sigma<\frac{n-d}{p} \qquad \text{or}\qquad \sigma=\frac{n-d}{p}\quad \text{ and }\quad q=\infty.
\]
For bounded Lipschitz domains $\Omega$, i.e., $d=n-1$, Theorem \ref{th:multipliers-2}  therefore yields $\chi_{\Omega}\in \mathbf{B}^{\sigma}_{p,q,\selfs}(\rn) $  \text{if} 
\beq\label{m-24}
\sigma<\frac 1p \qquad \text{or} \qquad \sigma=\frac 1p\quad \text{ and }\quad q=\infty.
\eeq
These  results are sharp since there exists a Lipschitz domain $\Omega$ in $\rn$ such that 
\[
\chi_{\Omega}\in \mathbf{B}^{\frac 1p}_{p,\infty,\selfs}(\rn)\qquad \text{and}\qquad \chi_{\Omega}\not\in \mathbf{B}^{\frac 1p}_{p,q}(\rn)\quad \text{if}\quad 0<q<\infty.
\]  
In order to see this let $\ \Omega=\left[-\frac 12,\frac 12\right]^{n}\ $. Observing that \[
\omega_r(\chi_{\Omega},t)_p\lesssim t^{\frac 1p}
\]
one calculates 
\[
\left(\int_0^1 t^{-\sigma q}\omega_r(\chi_{\Omega},t)^q_p\frac{\ud t}{t}\right)^{1/q}\lesssim \left(\int_0^1 t^{(\frac 1p-\sigma)q}\frac{\ud t}{t}\right)^{1/q}
\]
which is finite if, and only if, $\sigma$ satisfies \eqref{m-24}. 
Therefore, in view of Theorem \ref{th:multipliers-1},  concerning Lipschitz domains there is an \[\text{\em alternative s.t. either the trace of $\mathbf{B}^{\sigma}_{p,q}(\rn)$ on $\Gamma$ exists or $\chi_{\Omega}$ is a pointwise multiplier for $\mathbf{B}^{\sigma}_{p,q}(\rn)$},\] as was conjectured for F-spaces in  \cite[p.36]{tri02}: 
For smoothness $\sigma>\frac 1p$ we have traces according to Theorem \ref{B-trace-lip} whereas for $\sigma<\frac 1p$ we know that $\chi_{\Omega}$ is a pointwise multiplier for $\mathbf{B}^{\sigma}_{p,q}(\rn)$. The limiting case $\sigma=\frac 1p$ needs to be discussed separately: according to Corollary \ref{B-trace-lim-n} 
 we have traces for B-spaces with $q\leq \min(1,p)$,  
but $\chi_{\Omega}$ is (possibly) only a multiplier for $\mathbf{B}^{1/p}_{p,\infty}(\rn)$. There remains a 'gap' for spaces
\[
\mathbf{B}^{1/p}_{p,q}(\rn)\qquad \text{when}\qquad \min(1,p)<q<\infty. 
\] 
}

{\bf Acknowledgment:}\\

We thank Winfried Sickel and Hans Triebel for their valuable comments.
Jan Vyb\'\i ral acknowledges the financial support provided by the FWF project Y 432-N15 START-Preis ``Sparse Approximation
and Optimization in High Dimensions''.

\bibliographystyle{alpha}


\def\cprime{$'$} \def\cprime{$'$} \def\cprime{$'$} \def\cprime{$'$}
  \def\cprime{$'$} \def\cprime{$'$} \def\cprime{$'$} \def\cprime{$'$}


\end{document}